\documentclass[3p]{elsarticle}

\usepackage{mathrsfs}
\usepackage[dvipsnames,usenames]{color}
\usepackage{amsmath,amsthm}
\usepackage{amssymb}
\usepackage{graphicx}
\usepackage{subfigure}
\usepackage{times}
\usepackage{graphics,color}
\usepackage{tikz}
\usepackage{multirow}
\newtheorem{theorem}{Theorem}

\newtheorem{proposition}{Proposition}

\newtheorem{definition}{Definition}
\newtheorem{remark}{Remark}
\newcommand{\onetom}{1,\cdots,m}

\newcommand{\onetoN}{1,\cdots,N}

\newcommand{\R}{\mathbb R}
\begin{document}

\begin{frontmatter}
\title{Pinning networks of coupled dynamical systems with Markovian switching
couplings and event-triggered diffusions}
\author[fdu,fdub]{Wenlian Lu\corref{cor1}}
\ead{wenlian@fudan.edu.cn}
\author[smu]{Yujuan Han}
\ead{yjhan@shmtu.edu.cn}
\author[fdu,fduc]{Tianping Chen}
\ead{tchen@fudan.edu.cn}
\cortext[cor1]{Corresponding author.}
\address[fdu]{School of Mathematical Sciences, Fudan University, Shanghai 200433, China}
\address[smu]{College of Information Engineering, Shanghai Maritime University, Shanghai 201306, China}
\address[fdub]{Centre for Computational Systems Biology, Fudan University, Shanghai 200433, China}
\address[fduc]{School of Computer Science, Fudan University, Shanghai 200433, China}

\begin{abstract}
In this paper, stability of linearly coupled dynamical systems with feedback pinning algorithm is studied. Here, both the coupling matrix and the set of pinned-nodes vary with time, induced by a continuous-time Markov chain with finite states. Event-triggered rules are employed on both diffusion coupling and feedback pinning terms, which can efficiently reduce the computation load, as well as communication load in some cases  and be realized by the latest observations of the state information of its local neighborhood and the target trajectory. The next observation is triggered by certain criterion (event) based on these state information as well. Two scenarios are considered: the continuous monitoring, that each node observes the state information of its neighborhood and target (if pinned) in an instantaneous way, to determine the next triggering event time, and the discrete monitoring, that each node needs only to observe the state information at the last event time and predict the next triggering-event time. In both cases, we present several event-triggering rules and prove that if the conditions that the coupled system with persistent coupling and control can be stabilized are satisfied, then these event-trigger strategies can stabilize the system, and Zeno behaviors are excluded in some cases. Numerical examples are presented to illustrate the theoretical results.
\end{abstract}
\begin{keyword}
pinning control; event-triggered rule; coupled dynamical system; Markovian switching
\end{keyword}
\end{frontmatter}
\section{Introduction}
Control and synchronization of large-scale dynamical systems
have attracted increasing interests over the last several decades \cite{Liu12}-\cite{Lu07}. When a
networked system is unstable by itself, many control strategies are
designed to stabilize the networked system. Among them, pinning control is effective, because it is economically realized by controlling a partial of the nodes, instead of all nodes
in the network.
The general idea behind pinning control is that when applying some local
feedback controllers only to a fraction of nodes, the rest
of nodes can be propagated through the network interactions among nodes \cite{Xiang}-\cite{Chen07}.

More related to the present paper, for example in \cite{han}, the authors investigated the pinning control problem of coupled dynamical systems with Markovian switching couplings
and Markovian switching controller-set. In \cite{han} and most existing works in linearly coupled dynamical systems, each node needs to gather information of its own state and neighbors states and update them continuously or in a fixed sampling rate \cite{Lu04}. However, as pointed out in \cite{Astrom},
the event-based sampling technique showed better performance than sampling periodically in time for some simple systems. Hence, a number of researchers suggested that the event-based control algorithms can be utilized for the purpose to reduce communication and computation load in networked systems \cite{Tabuada}-\cite{Wang11} but still maintain control performance \cite{Wang11}-\cite{Yi1}. Therefore, the event-based control is particularly suitable for networked systems with limited resources and so has attracted wide interests in the scope of distributed control of networked systems. The idea of event-triggered control can be regarded as a specific temporal discretization approach, which was studied before \cite{Lu04,Lu07}. As application, the event-based control was applied for consensus of multi-agent systems. For instance,
\cite{Dimarogonas} investigated centralized and
distributed formulation of event-driven strategies for consensus of multi-agent systems and proposed a self-triggered setup; \cite{Johannesson,Rabi} studied the stochastic event-driven strategies;
\cite{Seyboth} introduced event-based control strategies for both networks of single-integrators with time-delay in communication links and networks of double-integrators; By using scattering transformation,
\cite{YuHan} studied the output synchronization problem of multi-agent systems with event-driven communication in the presence of constant communication delays.

In some recent papers \cite{Alderisio}-\cite{ZLiu}, the authors addressed event-triggered algorithms for pinning control of networks. \cite{Alderisio} gave an exponentially decreasing threshold function, while the event-triggering threshold in \cite{Gao,ZLiu} is prescribed by the continuous or discrete states of agents and target. In these works, sufficient conditions were proposed, which are based on the control gain, some quantities of the uncoupled node dynamics and the minimum eigenvalue of the augmented Laplacian.

Motivated by these works including our previous work \cite{han}, in the present paper, we employ the event-triggered strategy in both coupling configuration and pinning control terms to realize stability in coupled dynamical systems with Markovian switching couplings and pinned node set.
At each node, the diffusion coupling and feedback pinning terms are piecewise static based on the information of its local neighborhood and the target trajectory only at the latest time of event, which is triggered by some specified criteria derived from the information of its local neighborhood and target. In other words,
once the triggering criterion of node is satisfied, the diffusion coupling and pinning terms will be updated; otherwise these terms are constant between two successive event time points.
We consider two scenarios: continuous monitoring and discrete monitoring.
In the continuous-monitoring scenario, each node observes its neighborhood's and the target's states in an instantaneous way; on the contrary, in discrete-monitoring scenario, each node can only obtain its neighborhood's and target's information at the last event-triggering time point, which results in a small cost of monitoring (communication load) but the triggering events happening more frequently than continuous monitoring, namely, higher computation load. For both scenarios, it is proved that the proposed event-triggered rules guarantee the stability of the coupled dynamical systems under the local pinning algorithm.

This paper is organized as follows. After formulating the underlying problem in Sec. 2, the event-triggering rules of diffusion and pinning terms are proposed to pin
the coupled systems to a homogenous preassigned trajectory of the uncoupled node system by using continuous and discrete monitoring scenarios in Sec. 3 and 4 respectively. Simulations are given in Sec. 5 to verify the theoretical results. Strength, limitations of the work and possible orients of future study are discussed in Sec. 6.
Finally, this paper is concluded in Sec. 7.

\section{Problem formation}

In this paper, we consider a network of linearly coupled dynamical systems with discontinuous diffusions and feedback pinning terms as follows:
\begin{eqnarray}
\dot{x}_i(t)=f(x_{i}(t))+\theta_i(\sigma_t,t_k^i),~~~t^{i}_{k}\le t<t^{i}_{k+1},~i=\onetom,
\label{pinning_ds1}
\end{eqnarray}
where
\begin{eqnarray}
\theta_i(\sigma_t,t_k^i)=-c\sum_{j=1}^{m}L_{ij}(\sigma_t)\Gamma\bigg[x_{j}(t^{i}_{k})-x_{i}(t^{i}_{k})\bigg]
-c\epsilon D_i(\sigma_t)\bigg[x_{i}(t^{i}_{k})-s(t^{i}_{k})\bigg].
\label{pinning_ds2}
\end{eqnarray}

Here, the system contains $m$ nodes, $x_{i}(t) = [x_i^1(t),\cdots,x_i^n(t)]^{\top} \in\R^{n}$ denotes the state vector of node $i$, the continuous map $f(\cdot):\R^{n}\to\R^{n}$ denotes the identical node dynamics. $\sigma_t$ is a homogeneous Markov chain, which will be specified later. $L(\sigma_t)=[L_{ij}(\sigma_t)]_{i,j=1}^{m}\in\R^{m,m}$ is the time-varying Laplacian matrix of the underlying time-varying bi-graph $\mathcal G(\sigma_t)=\{V,E(\sigma_t)\}$, with the node set $V$ and time-varying link set $E(\sigma_t)$: for each pair of nodes $i\ne j$, $L_{ij}(\sigma_t)=-1$ if $i$ is linked to $j$ at time $t$, otherwise $L_{ij}(\sigma_t)=0$, and $L_{ii}(\sigma_t)=-\sum_{j=1}^{m}L_{ij}(\sigma_t)$. 
$c$ is uniform coupling strength at each node. $\Gamma=[\gamma_{ij}]_{i,j=1}^{n}\in\R^{n,n}$ stands for the inner configuration matrix that describes the coupling of components between the state vectors. 
$D_i(\sigma_t)=\delta_{\mathcal D(\sigma_t)}(i)$, where ${\delta}_{\cdot}(\cdot)$ is the characteristic function, i.e. $\delta_{\mathcal D(\sigma_t)}(i) = 1$ if $i \in\mathcal D(\sigma_t) $, otherwise $\delta_{\mathcal D(\sigma_t)}(i) = 0$ for the pinned node subset $\mathcal D(\sigma_t)\subset\{\onetom\}$, where
$\mathcal D(\sigma_t)$ denotes the node subset in $\mathcal V$ that are pinned at time $t$ by a specific node dynamic trajectory $s(t)$ with $\dot{s}=f(s(t))$, $s(0)=s_{0}$. $\epsilon$ is the pinning strength gain over the coupling strength.

Our aim is to provide sufficient conditions to guarantee that $s(t)$ is a global stable  trajectory  for the coupled system, namely
\begin{eqnarray*}
\lim_{t\to\infty}\|x_{i}(t)-s(t)\|=0,~\forall~i=\onetom.
\end{eqnarray*}
Here, we consider the $L_{2}$-vector norm and denote it by $\|\cdot\|$ throughout this paper.
Let $\hat{x}_{i}(t)=x_{i}(t)-s(t)$. Then (\ref{pinning_ds1} and \ref{pinning_ds2}) become
\begin{eqnarray}\label{pinning_ds3}
\dot{\hat{x}}_{i}(t)=f(\hat{x}_{i}(t)+s(t))-f(s(t))-c\sum_{j=1}^{m}L_{ij}(\sigma_t)
\Gamma\bigg[\hat{x}_{j}(t^{i}_{k})-\hat{x}_{i}(t^{i}_{k})\bigg]
-c\epsilon D_i(\sigma_t)\Gamma\hat{x}_{i}(t^{i}_{k}),~t^{i}_{k}\le t<t^{i}_{k+1}
\end{eqnarray}

Suppose $\sigma_{t}$ is a homogeneous continuous Markov chain with a finite
state space $\mathbb{S}=\{1,2,\cdots,N\}$ and its infinitesimal generator
$Q=[q_{uv}]_{N\times N}$ is given by
\begin{eqnarray*}
  \mathbb{P}\{\sigma_{t+\Delta}=v|\sigma_{t}=u\}
  =\left\{\begin{array}{lr}q_{uv}\Delta +o(\Delta), & u\ne v,\\
  1+q_{uu}\Delta+o(\Delta),& u=v,
  \end{array}
  \right.
\end{eqnarray*}
where $\Delta>0$, $o(\Delta)$ is a infinitesimal as $\Delta\to 0$, i.e,  $\lim_{\Delta \to 0}(o(\Delta)/\Delta)=0$,
$p_{uv}=-\frac{q_{uv}}{q_{uu}}>0$ is the transition probability from $u$ to
$v$ if $v\ne u$, while $q_{uu}=-\sum_{v=1,v\ne u}^{N}q_{uv}$. Denote $P=[p_{uv}]$ the transition matrix of the Markov chain. 
The sojourn time in state $u$ is exponentially distributed with parameter
$q_{u}\triangleq -q_{uu}$.

Let $\mathbb P_{t,u}(s,\Lambda)=\mathbb P(\sigma_{t+s}\in\Lambda|\sigma_t=u)$ and
$\mathbb E_{t,u} (f(\sigma_s))=\int f(y)\mathbb P_{t,u}(s,dy)$.
Denote by $\mathcal A$ the weak infinitesimal operator of $\sigma_t$.
A function $f(\sigma_t,t)$ is said to be
in the domain of $\mathcal A$ if
\begin{align*}
	&\lim_{\Delta\to 0}\frac{\mathbb E_{t,u} (f(\sigma_{t+\Delta},t+\Delta))-f(u,t)}{\Delta}\\
	=&\lim_{\Delta\to 0}\frac{\mathbb E_{t,u} (f(\sigma_{t+\Delta},t+\Delta))-\mathbb E_{t,u} (f(\sigma_{t+\Delta},t))}{\Delta}
+\lim_{\Delta\to 0}\frac{\mathbb E_{t,u} (f(\sigma_{t+\Delta},t))-f(u,t)}{\Delta}\\
=&f_{t}(u,t)+h(u,t)
\end{align*}
and
\begin{align*}
	\lim_{\Delta\to 0} \mathbb E_{t,u}(f_t(\sigma_{t+\Delta},t+\Delta)+h(\sigma_{t+\Delta},t+\Delta))
=f_t(u,t)+h(u,t).
\end{align*}
Then, we write $\mathcal A f(u,t)=f_t(u,t)+h(u,t)$.

Note that for fixed $t$,
\begin{align*}
	\mathbb E_{t,u} (f(\sigma_{t+\Delta},t))=\int f(y,t)\mathbb P_{t,u}(\Delta,dy)
=\sum_{v}f(v,t)\mathbb P_{t,u}(\Delta,v)
=\sum_{v}f(v,t)q_{uv}\Delta+f(u,t).
\end{align*}
Hence, by the Dynkin's formula, we have
\begin{align}\label{infinitesimal}
\mathcal A f(u,t)=f_t(u,t)+\sum_{v}f(v,t)q_{uv}.
\end{align}

Throughout the paper, we assume $f(\cdot,t)$ belongs to the following function
class.

\begin{definition}

Function class QUAD $(G,\alpha\Gamma,\beta)$: let $G$ be an $n\times n$ positive definite matrix and $\Gamma$ be an $n\times n$ matrix. QUAD$ (G,\alpha\Gamma,\beta)$ denotes a class of
continuous functions $f(\xi,t):\mathbb{R}^{n}\times[0,+\infty)\mapsto
\mathbb{R}^{n}$ satisfying
\begin{eqnarray*}\label{f}
(\xi-\zeta)^{\top}G[f(\xi,t)-f(\zeta,t)-\alpha \Gamma (\xi-\zeta)]
\le -\beta(\xi-\zeta)^{\top}G(\xi-\zeta)
\end{eqnarray*}
holds for all $\xi, \zeta \in \mathbb{R}^n$.
\end{definition}
\begin{definition}
System (\ref{pinning_ds1},\ref{pinning_ds2}) is said to be exponentially stable at $s(t)$ in mean
square sense, if there exists constants $\delta>0$ and $M>0$, such that
\begin{eqnarray}
\mathbb{E}\bigg[\|x_{i}(t)-s(t)\|^2\bigg]\leq M e^{-\delta t}
\end{eqnarray}
holds for all $t>0$ and any $i=1,\cdots,m$.
\end{definition}

\section{Continuous monitoring}

We briefly provide the basic idea of the setup of the coupling and pinning
terms. Instead of using the simultaneous state
from the neighborhood and the target trajectory to realize stability, an economic
alternative for the node $i$ is to use the neighbors' constant
states at the nearest time point $t_k^i$ until some pre-defined
event is triggered at time $t_{k+1}^i$; if node $i$ is pinned at time $t$, it also obtains the target trajectory's constant state at time point $t_{k_i(t)}^i$; then the incoming neighbors' and the target trajectory's information is updated by the states at $t_{k+1}^i$ until the next event is triggered, and so on. The event is defined based on
the neighbors', the target trajectory's and its own states with some prescribed rule.
This process goes on through all nodes in a parallel fashion.
To depict the event that triggers the next
time point, we introduce following Lyapunov function:
\begin{eqnarray}\label{V}
V(\hat{x},t,\sigma_t)&=&\frac{1}{2}\hat{x}^{\top}\bigg(P(\sigma_t)\otimes G\bigg)\hat{x},
\end{eqnarray}
where $\hat{x}=[\hat{x}_{1}^{\top},\cdots,\hat{x}_{m}^{\top}]^{\top}$, $P(\sigma_t)\in\R^{m,m}$ are diagonal positive definite matrices, induced by $\sigma_{t}$, and
$G\in\R^{n,n}$ is a positive definite matrix.
Let $\hat{F}(\hat{x})=[(f(x_{i})-f(s))^{\top},\cdots,(f(x_{m})-f(s))^{\top}]^{\top}$,
$D(\sigma_t)=diag[D_i(\sigma_t)]_{i=1}^{m}$,
$\hat{L}(\sigma_t)=L(\sigma_t)+\epsilon D(\sigma_t)$.

Note that $V(\hat{x},t,\sigma_t)$ is in the domain of the weak infinitesimal
operator of $\sigma_t$. Denoting $\sigma_t=u$, by (\ref{infinitesimal}), we have
\begin{eqnarray}\label{derivative_of_V}
\mathcal A V(\hat{x},t,u) = \sum_{v=1}^{N}q_{uv}V(\hat{x},t,v)+(\frac{\partial
V(\hat{x},t,u)}{\partial \hat{x}})^{\top}\frac{d{\hat{x}}}{dt}.
\end{eqnarray}
Substitute (\ref{pinning_ds3}) into (\ref{derivative_of_V}), we get
\begin{eqnarray}\label{derivative_of_V1}
\mathcal A V(\hat{x},t,\sigma_t) \left|_{(\ref{pinning_ds3})}\right.&=& \hat{x}^{\top}(t)\left[P(\sigma_t)\otimes G\right]\left[\hat{F}(\hat{x}(t))-\alpha I_m\otimes\Gamma \hat{x}(t)+ cz(t)\right]\\
\nonumber&+&\hat{x}^{\top}(t)\left\{P(\sigma_t)\left[(\alpha I_{m}-c \hat{L}(\sigma_t))\otimes G\Gamma\right]+\frac{1}{2}\sum_{v}q_{\sigma_t v} P(v)\otimes G\right\}^{sym}\hat{x}(t),
\end{eqnarray}
where $z(t)=[z_1^{\top}(t),\cdots,z^{\top}_m(t)]^{\top}$ with
\begin{eqnarray}\label{z_i}
z_i(t)&=&\sum_{j}L_{ij}(\sigma_t)\Gamma
\left[x_j(t)-x_i(t)-x_j(t_{k_i(t)}^i)+x_i(t_{k_i(t)}^i)\right]\nonumber\\
&+&\epsilon D_i(\sigma_t)\left[x_i(t)-s(t)-x_i(t_{k_i(t)}^i)-s(t_{k_i(t)}^i)\right],
\end{eqnarray}
and $\{S\}^{sym}$ denotes the symmetry part of a square matrix $S$, i.e., $S^{sym}=(S+S^{\top})/2$.

Let $\lambda_m(\cdot)$ and $\lambda_M(\cdot)$ denote the smallest and largest eigenvalues in module of a symmetry real matrix, and $\underline{\lambda}=\min_v\lambda_m\left(P(v)\otimes G\right)$, $\bar{\lambda}=\max_v\lambda_M\left(P(v)\otimes G\right)$. From the condition $f\in QUAD(P,\alpha\Gamma,\beta)$, we have
\begin{eqnarray}\label{V_term1}
\hat{x}^{\top}(t)\left[P(\sigma_t)\otimes G\right]\left[\hat{F}(\hat{x}(t))-\alpha I_m\otimes\Gamma \hat{x}(t)\right]\le
-\beta\underline{\lambda}\hat{x}^{\top}(t)\hat{x}(t).
\end{eqnarray}
For the term $\hat{x}^{\top}(t)\left[P(\sigma_t)\otimes G\right]z(t)$, we have
\begin{eqnarray}\label{V_term2}
\hat{x}^{\top}(t)\left[P(\sigma_t)\otimes G\right]z(t)&\le &
\frac{\upsilon}{2}\hat{x}^{\top}(t)(P^2(\sigma_t)\otimes G^2)\hat{x}(t)+\frac{1}{2\upsilon} z^{\top}(t)z(t)\\
\nonumber&\leq& \frac{\upsilon\bar{\lambda}^2}{2}
\hat{x}^{\top}(t)\hat{x}(t)+\frac{1}{2\upsilon}z^{\top}(t)z(t)
\end{eqnarray}
holds for any $\upsilon>0$.
Then, we have the following theorem.
\begin{theorem}\label{thm1}
Suppose that $f$ belongs to $QUAD(G,\alpha\Gamma,\beta)$ with the positive matrix $G$ and $\alpha>0, \beta>0$, and there exist diagonal
positive definite matrices $P(u), u=\onetoN$ such that
\begin{eqnarray}\label{slow_condition}
\left\{P(u)[\alpha I_{m}- cL(u)-c\epsilon D(u)]\otimes G\Gamma\right\}^{sym}
+\frac{1}{2}\sum_{v=1}^{N}q_{uv}P(v)\otimes G\leq 0,~\rm{for~ all}~ u\in \mathbb{S}.
\end{eqnarray}
Then, under either of the following two updating rules,  system (\ref{pinning_ds1}) is exponentially stable at the homogeneous trajectory $s(t)$
in mean square sense:
\begin{enumerate}
\item[(1)]
set $t^{i}_{k+1}$ by the rule
\begin{eqnarray}
t_{k+1}^{i}=\max\left\{\tau\ge t_{k}^{i}:~\|z_i(\tau)\|\le\frac{(\beta\underline{\lambda}-\frac{1}{2}\delta\bar{\lambda})}
{\sqrt{c}\bar{\lambda}}\|\hat{x}_i(\tau)\|\right\}
\label{event1}
\end{eqnarray}
where $0<\delta\le{2\beta\underline{\lambda}}/{\bar{\lambda}}$ is a constant;
\item[(2)]
set $t^{i}_{k+1}$ by the rule
\begin{eqnarray}
t_{k+1}^{i}=\max\left\{\tau\ge t_{k}^{i}:~\|z_i(\tau)\|\le a\exp{(-b\tau)}\right\}
\label{event2}
\end{eqnarray}
where $a>0$ and $b>0$ are constants.
\end{enumerate}
\end{theorem}
\begin{proof}

[Case (1).] Consider the event-triggering rule (\ref{event1}) and
pick a constant $\delta$ with $0<\delta\le{2\beta\underline{\lambda}}/{\bar{\lambda}}$. By Dynkin Formula \cite{Mao06}, we have
\begin{eqnarray}
\mathbb{E}e^{\delta t}V(\hat{x},t,\sigma_{t})
=\mathbb{E}V(\hat{x}(0),0,\sigma_0)+\delta\mathbb{E}\int_{0}^{t}e^{\delta \tau}V(\hat{x},\tau,\sigma_{\tau})d\tau
+\mathbb{E}\int_{0}^{t}e^{\delta \tau}\mathcal{A}V(\hat{x},\tau,\sigma_{\tau})d\tau.
\end{eqnarray}
From (\ref{derivative_of_V1})-(\ref{slow_condition}), we have
\begin{eqnarray}\label{dv}
\nonumber\mathbb{E}e^{\delta t}V(\hat{x},t,\sigma_{t})
\nonumber&\le& \mathbb{E}V(\hat{x}(0),0,\sigma_0)+\delta\mathbb{E}\int_{0}^{t}e^{\delta \tau}V(\hat{x},\tau,\sigma_{\tau})d\tau
-\beta\underline{\lambda}\mathbb{E}\int_{0}^{t}e^{\delta \tau}\hat{x}^{\top}(\tau)\hat{x}(\tau)d\tau\\
\nonumber&&+ \frac{c\upsilon\bar{\lambda}^2}{2}\mathbb{E}\int_{0}^{t}e^{\delta \tau}\hat{x}^{\top}(\tau)\hat{x}(\tau)d\tau+\frac{c}{2\upsilon}\mathbb{E}\int_{0}^{t}e^{\delta \tau}z^{\top}(\tau)z(\tau)d\tau\\
\nonumber&&+\mathbb{E}\int_{0}^{t}e^{\delta \tau}\hat{x}^{\top}(\tau)\left\{P(\sigma_{\tau})\left[(\alpha I_{m}-c \hat{L}(\sigma_{\tau}))\right]\otimes G\Gamma+\frac{1}{2}\sum_{v}q_{\sigma_{\tau} v} P(v)\otimes G\right\}^{sym}\hat{x}(\tau)d\tau\\
&\le& \mathbb{E}V(\hat{x}(0),0,\sigma_0)+\mathbb{E}\int_{0}^{t}e^{\delta \tau}\left\{\left[
-\beta\underline{\lambda}+\frac{\delta\bar{\lambda}}{2}+\frac{c\upsilon\bar{\lambda}^2}{2}\right]
\hat{x}^{\top}(\tau)\hat{x}(\tau)+\frac{c}{2\upsilon}z^{\top}(\tau)z(\tau)\right\}d\tau
\end{eqnarray}
for any $\upsilon>0$.
Note that
\begin{eqnarray*}
\max_{\upsilon>0}\frac{2\upsilon}{c}\left[\beta\underline{\lambda}-\frac{
\delta\bar{\lambda}}{2}-\frac{c\upsilon\bar{\lambda}^2}{2}\right]=
\frac{(\beta\underline{\lambda}-\frac{1}{2}\delta\bar{\lambda})^2}
{c\bar{\lambda}^2}
\end{eqnarray*}
and the maximum is reached if and only if $\upsilon=\frac{(\beta\underline{\lambda}-\frac{1}{2}\delta\bar{\lambda})}
{\sqrt{c}\bar{\lambda}}$. Hence, letting $\upsilon=\frac{(\beta\underline{\lambda}-\frac{1}{2}\delta\bar{\lambda})}
{\sqrt{c}\bar{\lambda}}$, (\ref{event1}) implies
\begin{eqnarray*}
\|z_i(\tau)\|^2\leq \frac{2\upsilon}{c}\left[\beta\underline{\lambda}-\frac{
\delta\bar{\lambda}}{2}-\frac{c\upsilon\bar{\lambda}^2}{2}\right]\|\hat{x}_i(\tau)\|^2,~~~
i=1,\cdots,m,
\end{eqnarray*}
for all $\tau\le t^{i}_{k+1}$. Therefore, we have
\begin{eqnarray}\label{exp_step1}
\mathbb{E}e^{\delta t}V(\hat{x},t,\sigma_{t})\le \mathbb{E}V(\hat{x}(0),0,\sigma_{0}),
\end{eqnarray}
which implies
\begin{eqnarray}\label{exp_step2}
\mathbb{E}e^{\delta t}\|x_{j}(t)-s(t)\|^{2}
\le \frac{2}{\underline{\lambda}}\mathbb{E}e^{\delta t}V(\hat{x},t,\sigma_{t})
\le \frac{2}{\underline{\lambda}}\mathbb{E}V(\hat{x}(0),0,\sigma_{0}),~~~\forall j=\onetom.
\end{eqnarray}
Therefore, we have
\begin{eqnarray}\label{exp_step3}
\mathbb{E}\|x_{j}(t)-s(t)\|^{2}\le \frac{2e^{-\delta t}}{\underline{\lambda}}\mathbb{E}V(\hat{x}(0),0,\sigma_{0}),~~~~\forall j=\onetom.
\end{eqnarray}

[Case (2).] Consider the event-triggering rule (\ref{event2}) and pick $\upsilon=\frac{2\beta\underline{\lambda}-\delta\bar{\lambda}}{c\bar{\lambda}^2}$. Then, we have
\begin{eqnarray*}
-\beta\underline{\lambda}+\frac{
\delta\bar{\lambda}}{2}+\frac{c\upsilon\bar{\lambda}^2}{2}=0.
\end{eqnarray*}
Substituting $\|z_{i}(\tau)\|\le a\exp(-b\tau)$ into (\ref{dv}) gives
\begin{eqnarray*}
\mathbb{E}e^{\delta t}V(\hat{x},t,\sigma_{t})&\le & \mathbb{E}V(\hat{x}(0),0,\sigma_{0})+\frac{a^2c^2\bar{\lambda}^2}
{2(2\beta\underline{\lambda}-\delta\bar{\lambda})}\mathbb{E}\int_{0}^{t}e^{(-2b+\delta) \tau}d\tau\\
&=&\mathbb{E}V(\hat{x}(0),0,\sigma_{0})+\frac{a^2c^2\bar{\lambda}^2}
{2(2\beta\underline{\lambda}-\delta\bar{\lambda})}\frac{1}{2b-\delta}
\left[1-e^{(-2b+\delta)t}\right].
\end{eqnarray*}
Let $C_0=\frac{a^2c^2\bar{\lambda}^2}
{2(2\beta\underline{\lambda}-\delta\bar{\lambda})}\frac{1}{2b-\delta}$. By the similar arguments as (\ref{exp_step1})-(\ref{exp_step3}), we have
\begin{eqnarray*}
\mathbb{E}\|x_{j}(t)-s(t)\|^{2}\le \frac{2e^{-\delta t}}{\underline{\lambda}}\left(\mathbb{E}V(\hat{x}(0),0,\sigma_{0})+
C_0\left[1-e^{(-2b+\delta)t}\right]\right)
\le  C_1e^{-\min(2b,\delta)t},~~~\forall j=\onetom,
\end{eqnarray*}
where $C_1=\frac{2}{\underline{\lambda}}\max(|\mathbb{E}V(\hat{x}(0),0,\sigma_{0})+C_0|,|C_0|)$. This completes the proof.
\end{proof}

\begin{remark}
If $\hat{x}(t_1)=0$ holds, then from (\ref{event1}), at time $t_1$, every node updated its feedback term and pinning term (if pinned). Therefore, from (\ref{pinning_ds1},\ref{pinning_ds2}), for any $t\ge t_1$, $\hat{x}(t)=0$ holds, which means $\hat{x}=0$ is an equilibrium of the system under the event-triggering rules.
\end{remark}

\begin{remark}
In fact, in our previous work \cite{han}, we studied pinning dynamic systems of networks with Markovian switching couplings and controller-node set.
By this theorem, we showed and proved that if the condition for the stability of the coupled  system with spontaneous coupling and control in \cite{han} can be satisfied, then the event-trigger strategies can stabilize the system too. Therefore, the issue of selection of pinned node in terms of guaranteeing stability is totally the same with the system with spontaneous diffusion and control.
\end{remark}



Under the updating rule (\ref{event2}), it can be proved that the Zeno behaviors \cite{Joh} is excluded by the arguments as the same fashion as in \cite{Alderisio}. While for the updating rule (\ref{event1}), similar to work \cite{Dimarogonas}, it should be pointed out that there exists at least one node such that its next inter-event interval is strictly positive.
\begin{proposition}\label{Zeno1}
Suppose that all hypotheses of Theorem \ref{thm1} hold. Under the updating rule (\ref{event1}), if the system does not reach stability, then there exists at least one node $i\in\{1,\cdots,m\}$ such that the next inter-event interval is strictly positive; Under the updating rule (\ref{event2}), if the system does not reach stability, the expectation of next inter-event interval of every node is strictly positive, further, it is lower bounded by some positive constant.
\end{proposition}
The proof is the similar to those in \cite{Alderisio,Dimarogonas} with some modifications. In fact, if at any time $t$, there exists one node $i$ such that (\ref{event1}) cannot hold as an equality for this node $i$. Hence, $t_{k+1}^{i}>t$ can be derived, which implies that the inter-event interval for node $i$ is positive. Otherwise, at time $t$, (\ref{event1}) holds as an equality for every node, that is, all nodes update their control law at this moment, which implies $z_i(t)=0$ holds for all $i$. However, since the network has not been stabilized at $s(t)$ yet, there exists at least node $j$ with $\hat{x}_{j}(t)\neq 0$, which implies $t^{j}_{k+1}>t$ holds, which implies that the next inter-event interval for node $j$ should be positive.

In comparison, under the rule (\ref{event2}), suppose $b<\frac{\delta}{2}$, for each node $i$, at $\tau=t^{i}_{k}$, we have $z_{i}(t^i_k)=0$ but the right-hand side of (\ref{event2}) is nonzero. Therefore, $t^{i}_{k+1}>t^{i}_{k}$ always holds and the low bound of the expectation of inter-event intervals can be estimated as
$\frac{1}{b}\log{\{1+\frac{1}{A+B}\}}$ with $L_f$ the Lipschtiz constant of $f(\cdot)$, $A = \frac{2mL_f+2c m(m+\epsilon)+L_f+cm}{ab}$ and $B =\frac{2m+1}{b}$.

\section{Discrete monitoring}

In the discrete monitoring scenario, each node $i$ can obtain its local neighborhood's state only at the time points $t^{i}_{k}$, $k=1,2,\cdots$. Meanwhile, if node $i$ is pinned, it can also obtain the target's state at latest time points $t_{k_i(t)}^i$. By this way, the rule to determine the next time point $t^{i}_{k+1}$ of obtaining state information only depends on the local states at $t^{i}_{k}$. In comparison, the triggering event rules (\ref{event1}) and (\ref{event2}) demand the instantaneous states after $t^{i}_{k}$.

Consider system (\ref{pinning_ds1},\ref{pinning_ds2}) and $V(x)$ as the candidate Lyapunov function with its derivative (\ref{derivative_of_V1}). We are to derive a triggering event rule from (\ref{event1}) in Theorem \ref{thm1}, which only depends on $t_{k}^{i}$. The estimations of the upper bounds of $\|z_i(t)\|$ and the lower bounds of $\|x_i(t)-s(t)\|$ for all $i$ are essential.

First of all, we take the switching time points of the Markov chain $\sigma_{t}$ to trigger the  state information updating for all nodes. Then, we are to estimate the upper-bound of $\|(x_{j}(t)-x_{j}(t_{k}^{i}))-(x_{i}(t)-x_{i}(t_{k}^{i}))\|$ with $L_{ij}(\sigma_t)\neq 0$, of which the evolution equation can be written as:
\begin{eqnarray}\label{evolution1}
\begin{cases}\frac{d{x}_{i}(t)}{dt}=f(x_{i}(t))+
\theta_{i}(\sigma_t,t_{k}^i)&\\
\frac{d{x}_{j}(t)}{dt}=f(x_{j}(t))+\theta_{j}(\sigma_t,t_{k_j(t)}^j)
&
\end{cases}
\end{eqnarray}
for $t^{i}_{k}\le t<\min\{t^{i}_{k+1},t^{j}_{k_{j}(t)+1}\} $ and initials $x_{i}(t^{i}_{k}), x_{j}(t^{i}_{k})$. Here, $\theta_{i}(\sigma_t,t_{k}^i)$ and $\theta_{j}(\sigma_t,t_{k_j(t)}^j)$ are constants in this time interval.

Let us consider a general form of  (\ref{evolution1}) as follows:
\begin{eqnarray}
\begin{cases}\frac{du}{dt}=f(u(t))+\theta&u(0)=u_{0}\\
\frac{dv}{dt}=f(v(t))+\vartheta&v(0)=v_{0},
\end{cases}\label{r2}
\end{eqnarray}
where $u,v,u_{0},v_{0}\in\R^{n}$.
Suppose that there exists a nonnegative-valued continuous map $\rho:\R_{\ge 0}\times\R^{4n}\to\R_{\ge 0}$ such that the solutions of (\ref{r2}) satisfy the following inequality:
\begin{eqnarray}
\|(u(t)-u_{0})-(v(t)-v_{0})\|\le\rho(t,\theta,\vartheta,u_{0},v_{0}).\label{r1}
\end{eqnarray}
Here the map $\rho$ depends on the node dynamics map $f(\cdot)$, the initial value $u_{0},v_{0}$ and inputs $\theta,\vartheta$, and satisfies  $\rho(0,\cdot,\cdot,\cdot,\cdot)\equiv 0$. Geometrically, $\rho$ is an upper-bound estimation of the difference between the two trajectories of (\ref{r2}) starting at $0$ of time-length $t$:
\begin{align*}
\|(u(t)-u_{0})-(v(t)-v_{0})\|
=\left\|\int_{0}^{t}[f(u(s))-f(v(s))]d s+(\theta-\vartheta) t\right\|.
\end{align*}

For example, if $f(\cdot)$ is Lipschitz (on the two trajectories): $\|f(u(s))-f(v(s))\|\le L_{f}\|u(s)-v(s)\|$ for all $s\ge 0$, then we have
\begin{align*}
\|(u(t)-u_{0})-(v(t)-v_{0})\|
\le &L_{f}\int_{0}^{t}\|(u(s)-u_{0})-(v(s)-v_{0})\|ds\\
&+(\|\theta-\vartheta\|+L_f\|u_{0}-v_{0}\|)~t.
\end{align*}
By the Gronwall-Bellman inequality \cite{Gron,Bell}, we have
\begin{align}
\label{Lips}
\|(u(t)-u_{0})-(v(t)-v_{0})\|\
\le\frac{(\|\theta-\vartheta\|+L_f\|u_{0}-v_{0}\|)}{L_{f}}[\exp(L_{f}t)-1].
\end{align}
We can take $\rho(t,\theta,\vartheta,u_{0},v_{0})$ as the right-hand side above, which equals to zero at $t=0$.

Second, we suppose that there exists a nonnegative map $\varrho:\R_{\ge 0}\times\R^{4n}\to\R_{\ge 0}$ such that the solutions of (\ref{r2}) satisfy:
\begin{eqnarray}
\|u(t)-v(t)\|\geq\varrho(t,\theta,\vartheta,u_{0},v_{0}).
\end{eqnarray}
Here, $\varrho$ can be regarded as the lower-bound estimation of the distance between two trajectories:
\begin{align*}
\|u(t)-v(t)\|
=\left\|\int_{0}^{t}[f(u(s))-f(v(s))]ds
+(\theta-\vartheta)t
+(u_0-v_0)\right\|
\end{align*}
and satisfies (i). $\varrho(\cdot,\theta,\theta,u_{0},u_{0})\equiv 0$; (ii). $\varrho(0,\cdot,\cdot,u_{0},u_{0})\equiv 0$.
For example, assuming that there exists some constant $\sigma$ (possibly negative) such that
\begin{eqnarray*}
(u-v)^{\top}(f(u)-f(v))\ge\sigma (u-v)^{\top}(u-v)
\end{eqnarray*}
holds for all $u,v\in\mathbb R^{n}$.
We have
\begin{align*}
&\frac{d}{dt}[(u(t)-v(t))^{\top}(u(t)-v(t))]\left|_{(\ref{r2})}\right.
=2(u-v)^{\top}[f(u)-f(v)+\theta-\vartheta]\\
&\ge{2\sigma}(u-v)^{\top}(u-v)-\mu (u-v)^{\top}(u-v)
-\frac{1}{\mu}(\theta-\vartheta)^{\top}(\theta-\vartheta)
\end{align*}
hold for any $\mu>0$.
By Gronwall-Bellman inequality, we have
\begin{align*}
(u(t)-v(t))^{\top}(u(t)-v(t))
\ge&\exp{[({2\sigma}-\mu)t]}(u_{0}-v_{0})^{\top}(u_{0}-v_{0})\\
&-\frac{(\theta-\vartheta)^{\top}
(\theta-\vartheta)/\mu}{2\sigma-\mu}\bigg\{\exp[({2\sigma}-\mu)t]-1\bigg\}.
\end{align*}
We take $\varrho$ as the right-hand side above, which is positive for a small interval of $t$, starting from $0$, for any $u_{0}\ne v_{0}$.

We highlight that there is no uniform approach to get precise estimation for a general function of $f(\cdot)$ but one can do it case by case. Therefore, an efficient way is to use integrators that simulates the node dynamics of $\dot{u}=f(u)+\theta$ to realize generators that calculate the maps of $\rho$ and $\varrho$. Noting that these generators are independent of the states of the nodes, they can be built parallel to the networked system. Figures \ref{rho}  and \ref{varrho} show the generators of $\rho$ and $\varrho$ respectively.

\begin{figure}[!t]
\begin{center}
\includegraphics[height=.4\textwidth,width=.45\textwidth]{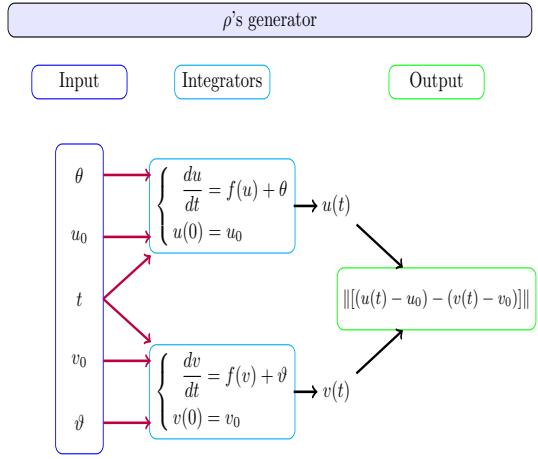}
\caption{$\rho$ generator.} \label{rho}
\end{center}
\end{figure}

\begin{figure}[!t]
\begin{center}
\includegraphics[height=.4\textwidth,width=.45\textwidth]{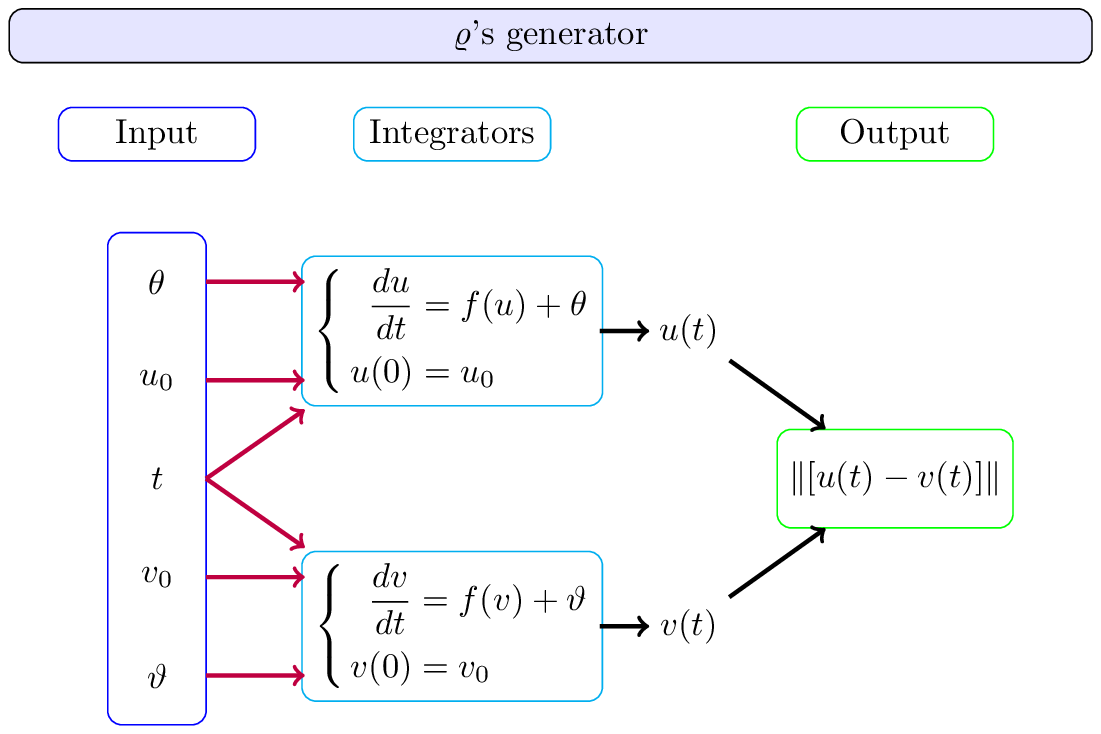}
\caption{$\varrho$ generator.} \label{varrho}
\end{center}
\end{figure}

Let $\vartheta^{i}_{k} = \theta_i(\sigma_t,t_k^i)$ and $\vartheta_{k_{j}(t)}^{j} = \theta_j(\sigma_t,t_{k_j(t)}(t))$.
Based on the event-triggering rules (\ref{event1}), (\ref{event2}), we have the following theorem.
\begin{theorem}\label{thm2}
Suppose that $f$ belongs to $QUAD(G,\alpha\Gamma,\beta)$ with positive matrix $G$ and $\alpha, \beta>0$. Suppose there exist diagonal
positive definite matrices $P(u), u=\onetoN$ such that
\begin{eqnarray*}
\left\{P(u)[\alpha I_{m}- cL(u)-c\epsilon D(u)]\otimes G\Gamma\right\}^{sym}
+\frac{1}{2}\sum_{v=1}^{N}q_{uv}P(v)\otimes G\leq 0,~\rm{for~ all}~ u\in \mathbb{S}.
\end{eqnarray*}
define a sequence of $\xi^{i}_{k}$ under either of the following two updating rules,
\begin{enumerate}
\item[(1)]
\begin{eqnarray}
\xi^{i}_{k}=\max\bigg\{&\xi:&~\sum_{j\ne i}(-L_{ij}(\sigma_{\xi+t_k^i}))\rho\left(\xi,{\vartheta}^{i}_k,
{\vartheta}^{j}_{k_j(\xi+t_k^i)},x_{i}(t^{i}_{k}),x_{j}(t^{i}_{k})\right)+\nonumber\\
&&+\epsilon D_i(\sigma_{\xi+t_k^i})\rho\left(\xi,{\vartheta}^{i}_k,
0,x_{i}(t^{i}_{k}),s(t^{i}_{k})\right)\nonumber\\
&&\le\frac{(\beta\underline{\lambda}-\frac{1}{2}\delta\bar{\lambda})}
{\sqrt{c}\bar{\lambda}}\varrho\left(\xi,\vartheta^{i}_k,
0,x_{i}(t^{i}_{k}),s(t^{i}_{k})\right)\bigg\}\label{event3}
\end{eqnarray}
where $0<\delta\le{2\beta\underline{\lambda}}/{\bar{\lambda}}$ is a constant;
\item[(2)]
\begin{eqnarray}
\xi^{i}_{k}=\max\bigg\{&\xi:&~\sum_{j\ne i}(-L_{ij}(\sigma_{\xi+t_k^i}))\rho\left(\xi,{\vartheta}^{i}_k,
{\vartheta}^{j}_{k_j(\xi+t_k^i)},x_{i}(t^{i}_{k}),x_{j}(t^{i}_{k})\right)+\nonumber\\
&&+\epsilon D_i(\sigma_{\xi+t_k^i})\rho\left(\xi,{\vartheta}^{i}_k,
0,x_{i}(t^{i}_{k}),s(t^{i}_{k})\right)\nonumber\\
&&\le a\exp{(-b(\xi+t_k^i))}\bigg\}\label{event4}
\end{eqnarray}
where $a>0$ and $b>0$ are constants.
\end{enumerate}
If the triggering event time points $\{t^{i}_{k}\}$ are picked by the following scheme:
\begin{enumerate}
\item Initialization: $t_{0}^{i}=0$ for all $i=\onetom$;
\item At $t=t^{i}_{k}$, node $i$ obtains $\xi^{i}_{k}$ by the rule (\ref{event3})(or (\ref{event4}));
\item 
At $t>t^{i}_{k}$,  if one of its neighbor, for example, denoted by $j$, is triggered at $t=t^{j}_{k'+1}$ (let $k'$ be the latest event at node $j$ before $t$), then $j$ broadcasts its current updating law, $\vartheta^{j}_{k'}$, to node $i$, and the rule (\ref{event3})(or (\ref{event4})) is updated by replacing the diffusion term from node $j$, $\vartheta^{j}_{k'}$, by $\vartheta^{j}_{k'+1}$, and $t^{i}_{k}$ by $t$. And, go to Step 2;

\item If $\sigma_{t}$ switches at $t$, then we update the rule (\ref{event3})(or (\ref{event4})) by replacing $\vartheta^{i}_{k}$ by the current state $\theta_i(\sigma_t,t)$, and $t^{i}_{k}$ by $t$. And go to Step 2;
\item Let $t^{i}_{k+1}=t^{i}_{k}+\xi^{i}_{k}$, an event is triggered at node $i$ by updating the state information in (\ref{pinning_ds1},\ref{pinning_ds2}) from $t^{i}_{k}$ by $t^{i}_{k+1}$,
\end{enumerate}
then system (\ref{pinning_ds1},\ref{pinning_ds2}) is stabilized at $s(t)$ in mean square sense.
\end{theorem}

This theorem can be derived from Theorem \ref{thm1} immediately. In fact,  event (\ref{event3}) is an estimation of event (\ref{event1}), event (\ref{event4}) is an estimation of event (\ref{event2}).

There is substantial difference between the discrete and continuous monitoring strategies. Generally speaking, the continuous monitoring require that every node collects its neighborhood states at every instant time, while discrete monitoring does not need this step. As shown in Table \ref{table1}, the continuous monitoring scheme costs higher communication load than the discrete monitoring. As a pay-off, we will show in the numerical example section that the frequencies of triggering events in the continuous monitoring are much lower than that the discrete monitoring requires. That is, the continuous monitoring costs lower computation load than the discrete monitoring.

\begin{table}[h]
\centering
\caption{Continuous \emph{vs} discrete time monitoring schemes\label{table1}}
\begin{tabular}{|r|p{.45\textwidth}|p{.45\textwidth}|}
\hline
Step&Continuous monitoring & Discrete-time monitoring\\
\hline
\multirow{2}{*}{1 }
& At time $t_k^i$, agent $i$ updates feedback control law $\theta_i(\sigma_t,t_k^i)$
& At time $t_k^i$, agent $i$ updates feedback control law $\theta_i(\sigma_t,t_k^i)$
\\
\hline
\multirow{1}{*}{2 }
& If  $t<t_{k+1}^i$ in (\ref{event1}) or (\ref{event2}) & If $\xi<\xi_k^i$ in (\ref{event3}) or (\ref{event4})
\\
\hline
\multirow{4}{*}{3}
& then
&\\
& monitoring the states of $i$'s neighborhood $x_j(t)$, $j\in \mathcal N_i$ and target $s(t)$ (if $i$ is pinned at time $t$), $t\ge t_k^i$
&
\\
\hline
\multirow{2}{*}{4 }
&else
&else\\
&
go to step 1, replace $t_k^i$ by $t_{k+1}^i$
& go to step 1, replace $t_k^i$ by $t_k^i+\xi_k^i$
\\
\hline
\end{tabular}
\end{table}

\begin{remark}
In the discrete monitoring scenario, each node does not need to observe the information of its neighbors at every instants, but each node has to broadcast its updating law, $\theta^{i}_{k}$, to all its neighborhood once it is triggered.
\end{remark}


Similar to Proposition \ref{Zeno1}, we have:
\begin{proposition}\label{Zeno2}
\begin{itemize}
\item [(1)]
Suppose that hypotheses in Theorem \ref{thm2} hold. Under the rule (\ref{event3}) and the scheme described in Theorem \ref{thm2}, if system (\ref{pinning_ds1},\ref{pinning_ds2}) is not stable at $t$, there exists at least one node $i\in\{1,\cdots,m\}$ such that the next triggering event time strictly greater than $t$, namely, inter-event interval is strictly positive.
\item [(2)] Under the rule (\ref{event4}) and the scheme described in Theorem \ref{thm2}, if system (\ref{pinning_ds1},\ref{pinning_ds2}) does not reach stability, 
     the expectation of next inter-event interval of every node is strictly positive, further, it is lower bounded by some positive constant.
\end{itemize}
\end{proposition}
\begin{remark}
The discrete monitoring strategy implies the triggering events happen more frequently than continuous monitoring as a reward of a smaller cost of monitoring.
\end{remark}

\begin{remark}
For the discrete monitoring strategy, the computation complexity for every task depends on the number of multiplies in $\rho(\cdot)$ and $\varrho(\cdot)$. We suppose the number of multiplies of $\rho(\cdot)$ and $\varrho(\cdot)$ are respectively $N_1$, $N_2$. From the updating rules (\ref{event3}) and (\ref{event4}), the computation complexity for the next triggering time of every agent is at most $(m+1)N_1+N_2$.
\end{remark}

\section{Examples}
In this section, we present several numerical examples to illustrate these theoretical results. The system is an array of $5$ coupled Chua circuits with the map $f(\cdot)$ of node dynamics as follows:
\begin{eqnarray}
f(z)=\left[\begin{array}{c}p*(-z_{1}+z_{2}-g(z_{1}))\\
z_{1}-z_{2}+z_{3}\\
-q*z_{2}\end{array}\right]
\end{eqnarray}
where $g(z_{1})=m_{1}*z_{1}+1/2*(m_{0}-m_{1})*(|z_{1}+1|-|z_{1}-1|)$, with the parameters taken values as $p=9.78$, $q=14.97$, $m_{0}=-1.31$ and $m_{1}=-0.75$, which implies that the intrinsic node dynamics (without coupling terms) have a double-scrolling chaotic attractor \cite{chua}. Let $P=\Gamma=G=I_{3}$, where $I_{3}$ stands for the identity matrix of three dimensions. To estimate the parameter $\beta$ in the $QUAD$ condition, noting the Jacobin matrices of $f$ is one of the following
\begin{eqnarray*}
A_{1}=\left[\begin{array}{lll}-2.445&9.78&0\\
1&-1&1\\0&-14.97&0\end{array}\right],A_{2}=\left[\begin{array}{lll}3.0318&9.78&0\\
1&-1&1\\0&-14.97&0\end{array}\right],
\end{eqnarray*}
then we estimate $\beta'=\alpha-\lambda_{\max}((A_{2})^{s})=\alpha-9.1207$, where $9.1207$ is the largest eigenvalue of the symmetry parts of all Jacobin matrices of $f$.

The possible coupling graph topologies are shown in Fig.\ref{network_topology}. To select the pinned nodes, we add an extra virtual node (on behalf of $s(t)$) to the original network, which has a few links to the node that are pinned, then have an {\em extended network}. From the results in \cite{han}, one can see that if every extended network topology among the switching is strongly connected and the duration time at each network topology is sufficiently long, then there always exist positive matrices $P(\cdot)$, coupling gain $c$ and pinning gain $\epsilon$ such that condition (\ref{slow_condition}) holds, which implies that the system with persistent coupling and control can be stabilized. Hence, according to this viewpoint, one node from every strongly connected component is picked to be pinned. Noting topologies given in  Fig.\ref{network_topology} are all connected, every possible pinned node set is applicable. 

In this simulation, the graph topologies and pinned nodes of coupled system (\ref{pinning_ds1},\ref{pinning_ds2}) switch among these four states, as shown in Fig.\ref{network_topology} (a)-(d) respectively, induced by a homogeneous Markov chain, $\sigma_t$. $L(\sigma_t)$ is picked as the Laplacian of the graph $\mathcal G(\sigma_t)$, where each link has uniform weight $1$. Here, we pick the state space of the Markov chain $\sigma_t$ by $\{1,2,3,4\}$, and its transition matrix is given by
\begin{eqnarray*}
T=\left[
\begin{array}{cccc}
  -10 & 6.5 & 0 & 3.5 \\
  7 & -10 & 3 & 0  \\
  0 & 1 & -10 & 9\\
  4 & 6 & 0 & -10
\end{array}
\right].
\end{eqnarray*}
It can be seen from the transition matrix of $\sigma_t$, the expected sojourn time in each graph follows an exponential distribution with parameter $0.1$.
In the following, we pick $\alpha=10$ and $\beta=0.8803.$ 

The ODEs (\ref{pinning_ds1},\ref{pinning_ds2}) are numerically solved by the Euler method with time step $0.001$ (sec) and the time duration of the numerical simulations is $[0,10]$ (sec).

\begin{figure}[!t]
\centering
\subfigure[Pinned set $\{2,6\}$]{
\includegraphics[height=.35\textwidth,width=.4\textwidth]{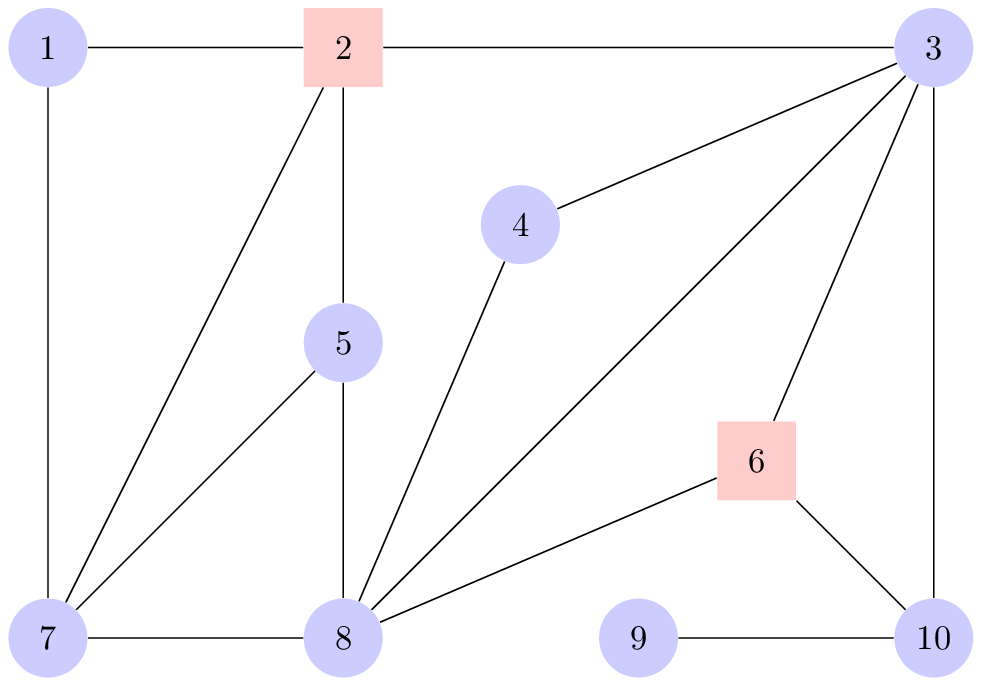}}
\subfigure[Pinned set $\{5,8\}$]{\includegraphics[height=.35\textwidth,width=.4\textwidth]{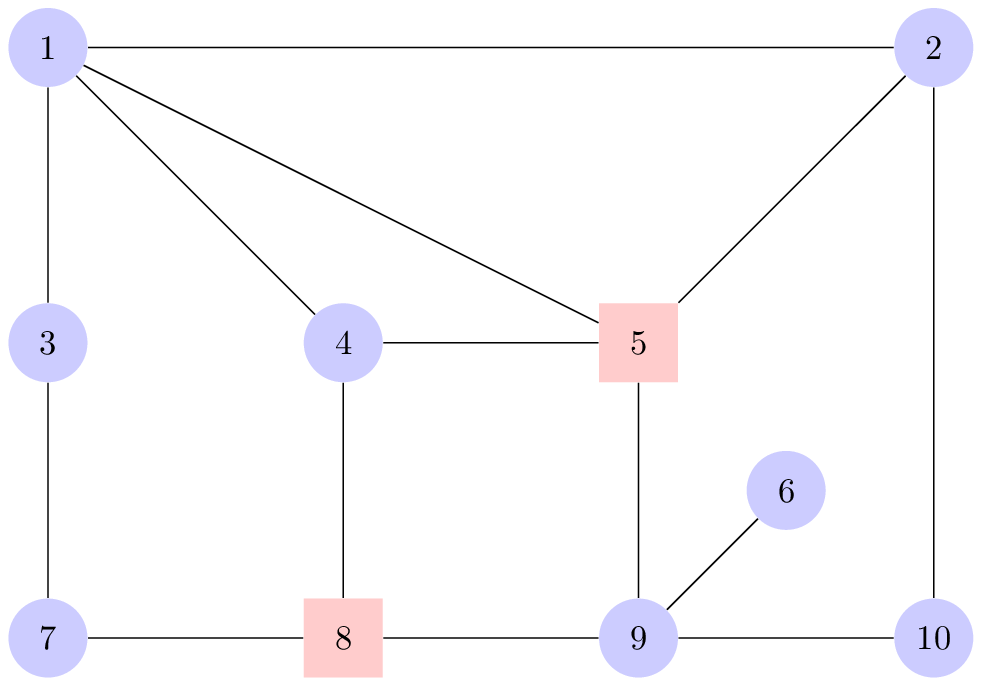}}
\subfigure[Pinned set $\{2,6\}$]{\includegraphics[height=.35\textwidth,width=.4\textwidth]{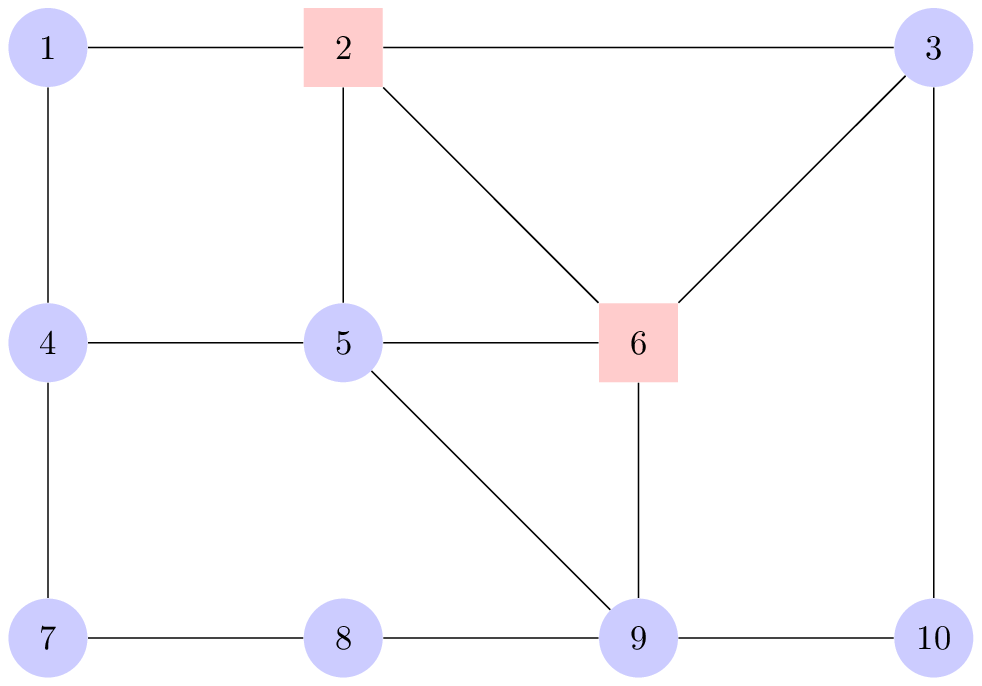}}
\subfigure[Pinned set $\{2,5\}$]{\includegraphics[height=.35\textwidth,width=.4\textwidth]{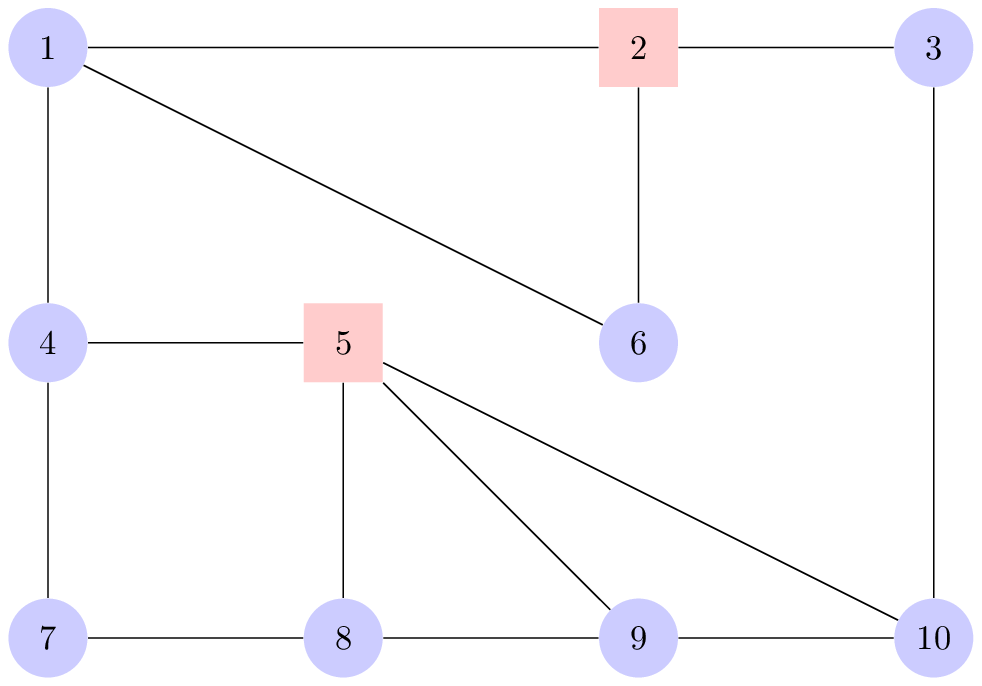}}
\label{network_topology}
\caption{The topologies of the graph of the coupled system and the pinned set.}
\end{figure}

\subsection{Continuous monitoring}
We give two examples to illustrate the updating rules (\ref{event1}) and (\ref{event2}) respectively. Pick $\delta =0.03$, $c=20$, $a = b =0.5$ and $\epsilon=0.5$. Then, it can be verified that $\{\alpha I_{m}-kL(u)-k\epsilon D(u)\}^s$ are negative definite and so the matrix inequality (\ref{slow_condition}) is satisfied.  And we have $\frac{(\beta\underline{\lambda}-\frac{1}{2}\delta\bar{\lambda})}
{\sqrt{c}\bar{\lambda}}=0.2736$.

We employ rule (\ref{event1}). Fig. \ref{fig_variation_x_con1} shows the dynamics of each components of the $10$ nodes and Fig. \ref{fig_variation_v} shows the dynamics of $V(t)$. All show that the coupled system (\ref{pinning_ds1}, \ref{pinning_ds2}) is stable. Similarly, we also employ the triggering event rule (\ref{event2}).  Fig. \ref{fig_variation_x_con2} illustrates the dynamics of each components of all nodes, and Fig. \ref{fig_variation_v} illustrates the dynamics of $V(t)$. One can see that the coupled systems (\ref{pinning_ds1},\ref{pinning_ds2}) is asymptotically stable at certain chaotic homogeneous trajectory. 

\begin{figure}[!t]
\centering
\includegraphics[height=.6\textwidth]{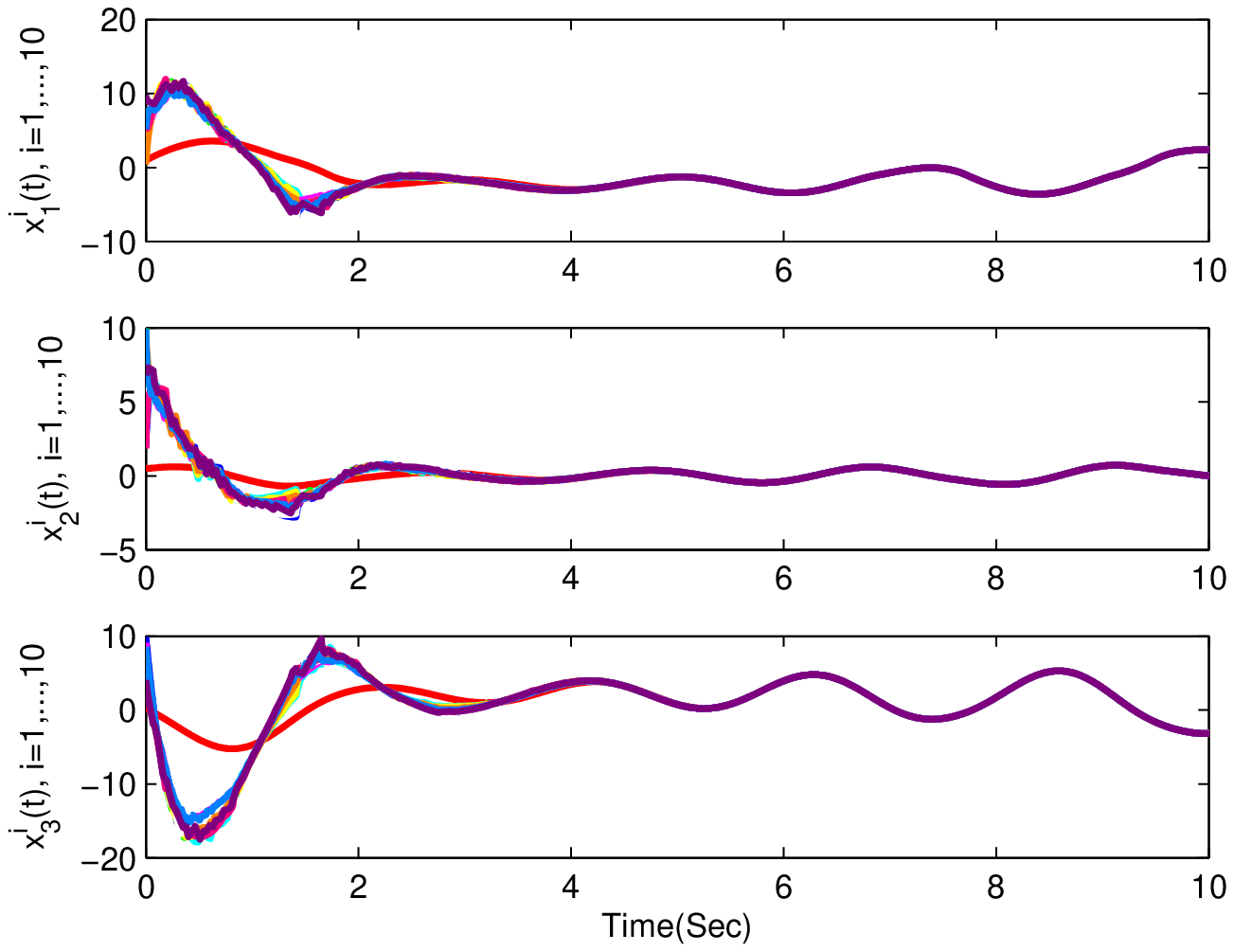} 
\caption{For  continuous monitoring with event triggering (\ref{event1}), the dynamics of components of the coupled system ((\ref{pinning_ds1},\ref{pinning_ds2})).}
\label{fig_variation_x_con1}
\end{figure}

\begin{figure}[!t]
\centering
\includegraphics[height=.6\textwidth]{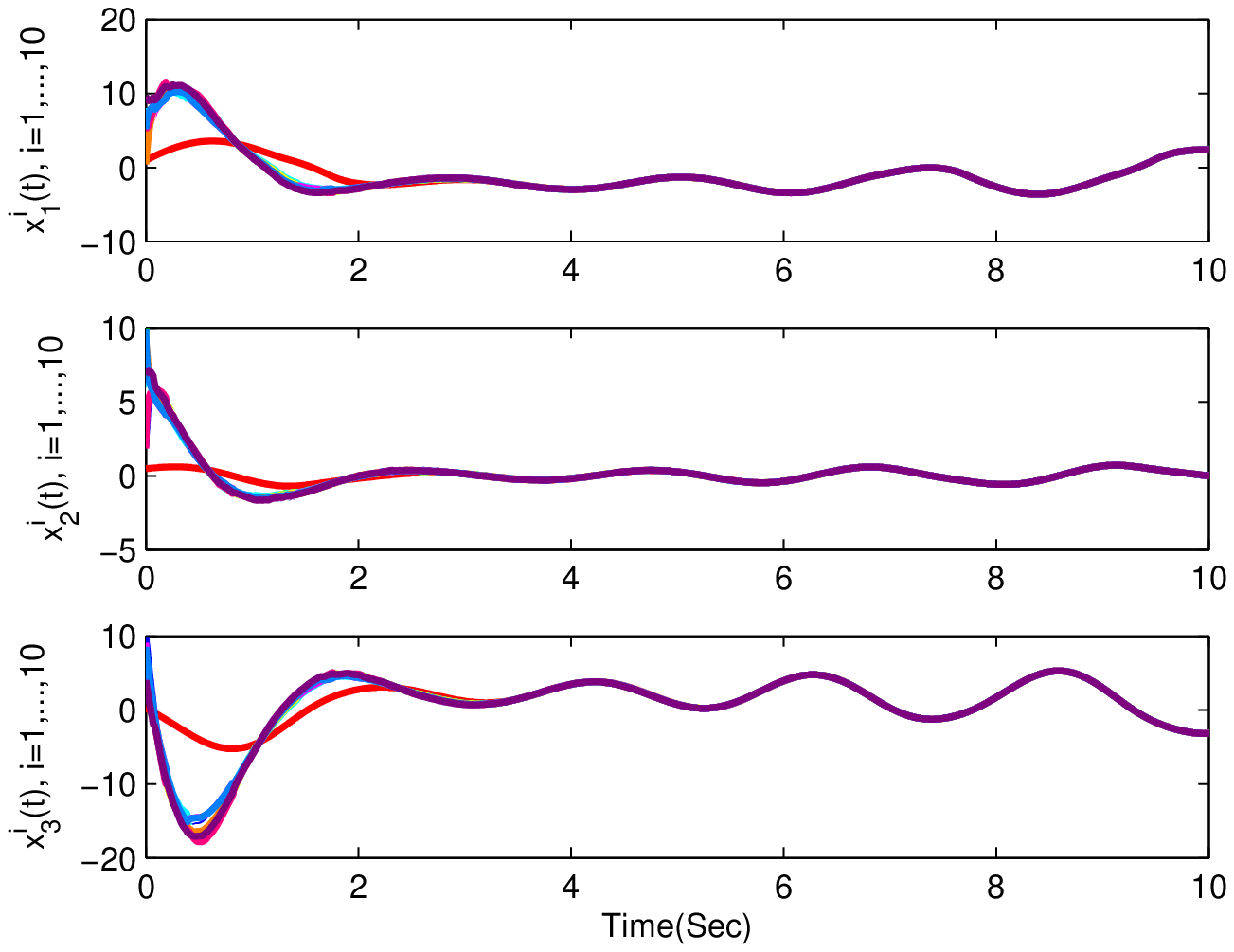} 
\caption{For  continuous monitoring with event triggering (\ref{event2}), the dynamics of components of the coupled system ((\ref{pinning_ds1},\ref{pinning_ds2})).}
\label{fig_variation_x_con2}
\end{figure}

\subsection{Discrete monitoring}

In this subsection, we illustrate the discrete-time monitoring strategies as described in Theorem \ref{thm2}. In these examples, we also take $c=20$, $a = b = 0.5$, and $\epsilon=0.5$. Fig.\ref{fig_variation_x_dis1} shows the dynamics of each components of the $10$ nodes and Fig.\ref{fig_variation_v} shows the dynamics of $V(t)$  under rule (\ref{event3}) in Theorem \ref{thm2}. All of them show that the coupled system (\ref{pinning_ds1},\ref{pinning_ds2}) is stable. By employing  rule (\ref{event4}) in Theorem \ref{thm2}, Fig. \ref{fig_variation_x_dis2} illustrates the dynamics of each components of all nodes and Fig. \ref{fig_variation_v} illustrates the dynamics of $V(t)$. These plots show that the coupled system (\ref{pinning_ds1},\ref{pinning_ds2}) is asymptotically stable. 

\begin{figure}[!t]
\centering
\includegraphics[height=.6\textwidth]{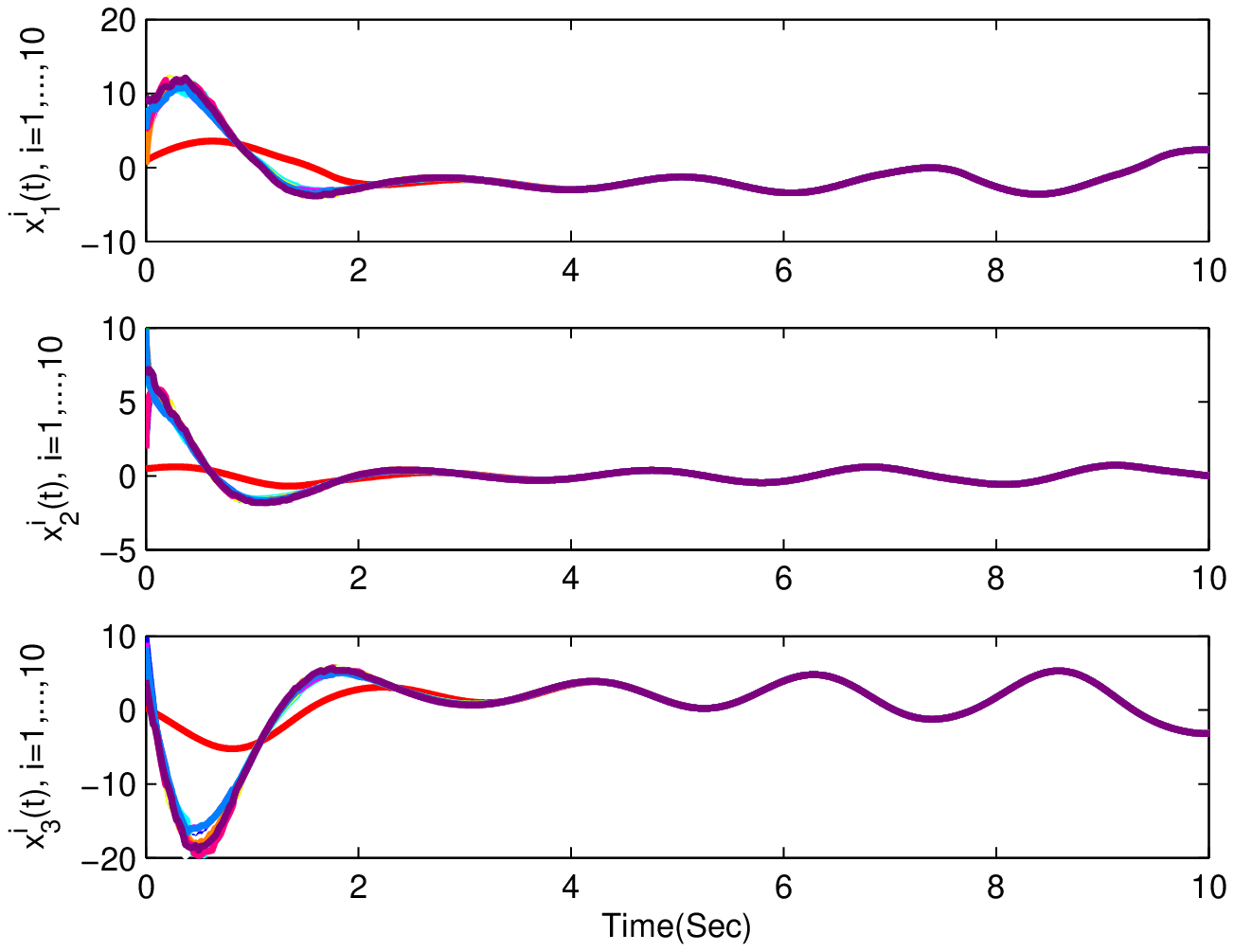} 
\caption{For  discrete monitoring with event triggering (\ref{event3}), the dynamics of components of the coupled system ((\ref{pinning_ds1},\ref{pinning_ds2})).}
\label{fig_variation_x_dis1}
\end{figure}

\begin{figure}[!t]
\centering
\includegraphics[height=.6\textwidth]{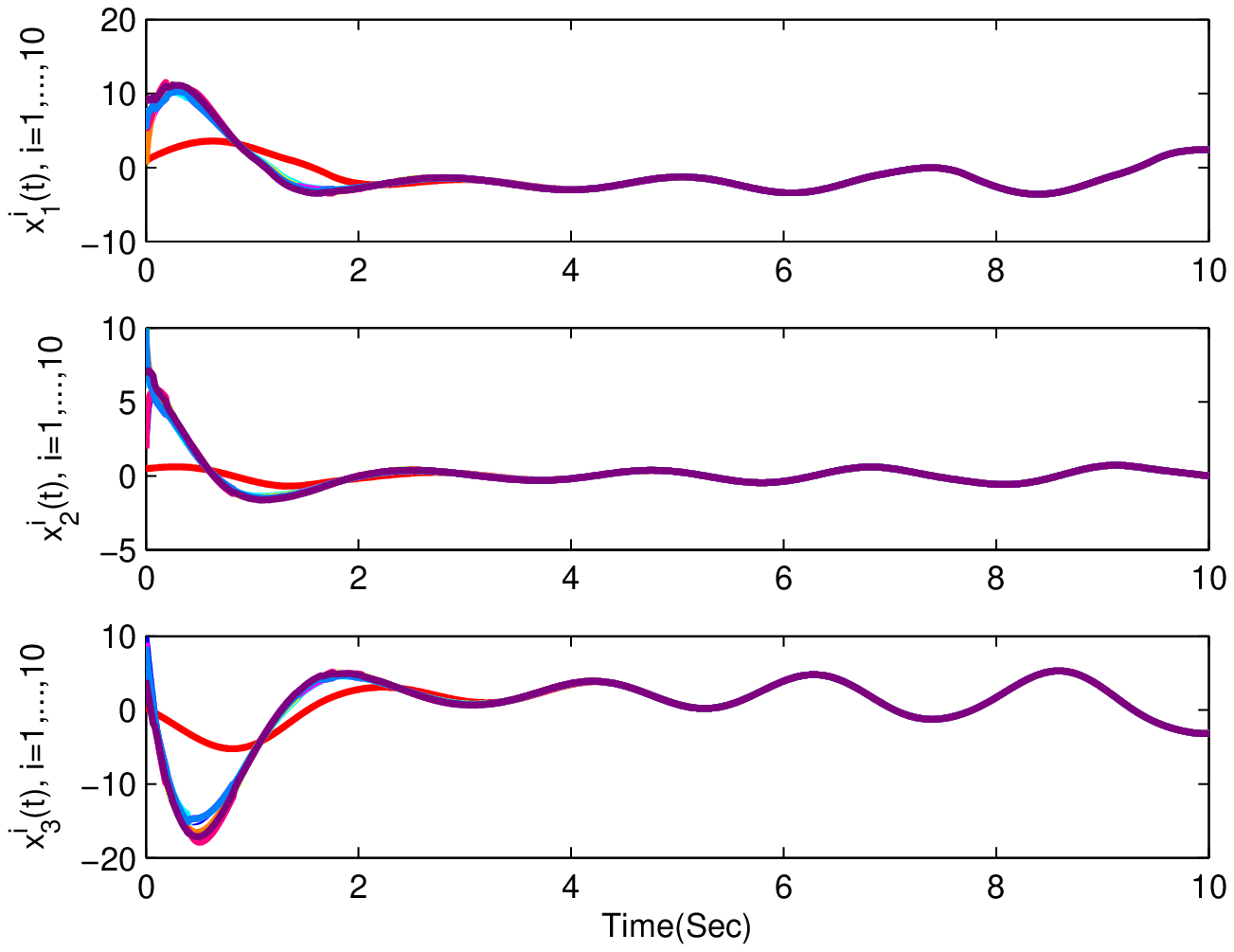} 
\caption{For  discrete monitoring with event triggering (\ref{event4}), the dynamics of components of the coupled system ((\ref{pinning_ds1},\ref{pinning_ds2})).}
\label{fig_variation_x_dis2}
\end{figure}

\begin{figure}[!t]
\begin{center}
\includegraphics[height=.4\textwidth]{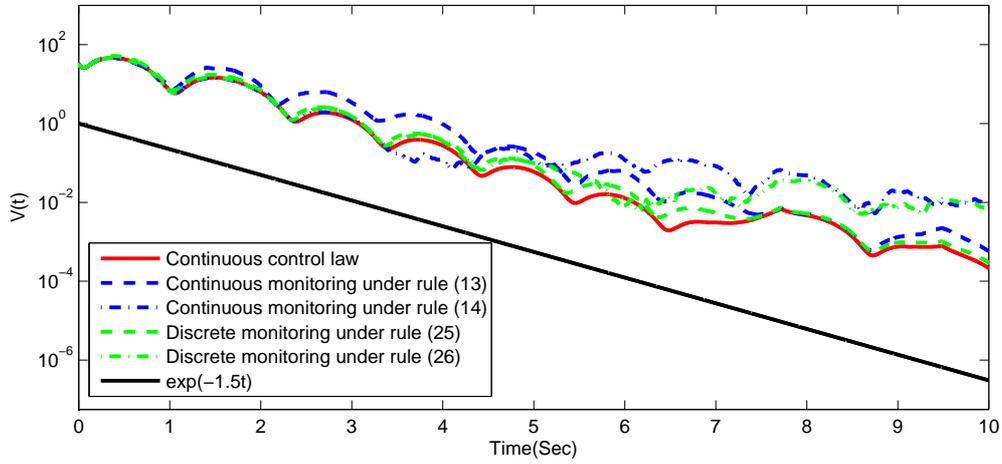}
\caption{The dynamics of Lyapunov function $V(t)$ for systems with continuous control law, continuous monitoring with rules (\ref{event1}),(\ref{event2}), discrete monitoring with rules (\ref{event3}), (\ref{event4}). } \label{fig_variation_v}
\end{center}
\end{figure}

\begin{figure}[!t]
\centering
\subfigure[]{
\includegraphics[width=.8\textwidth]{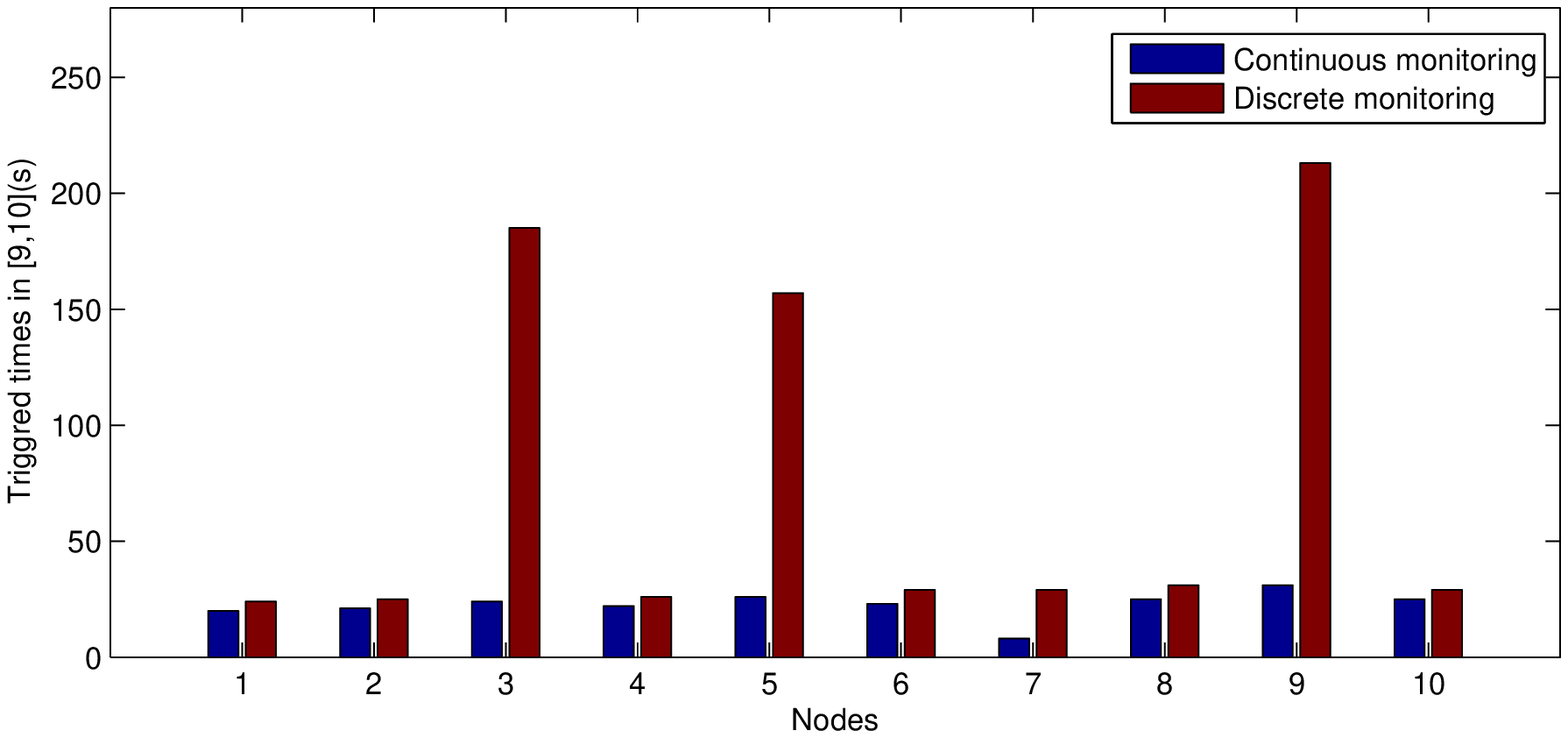} 
\label{fig_con_vs_dis1}
}
\subfigure[u]{
\includegraphics[width=.8\textwidth]{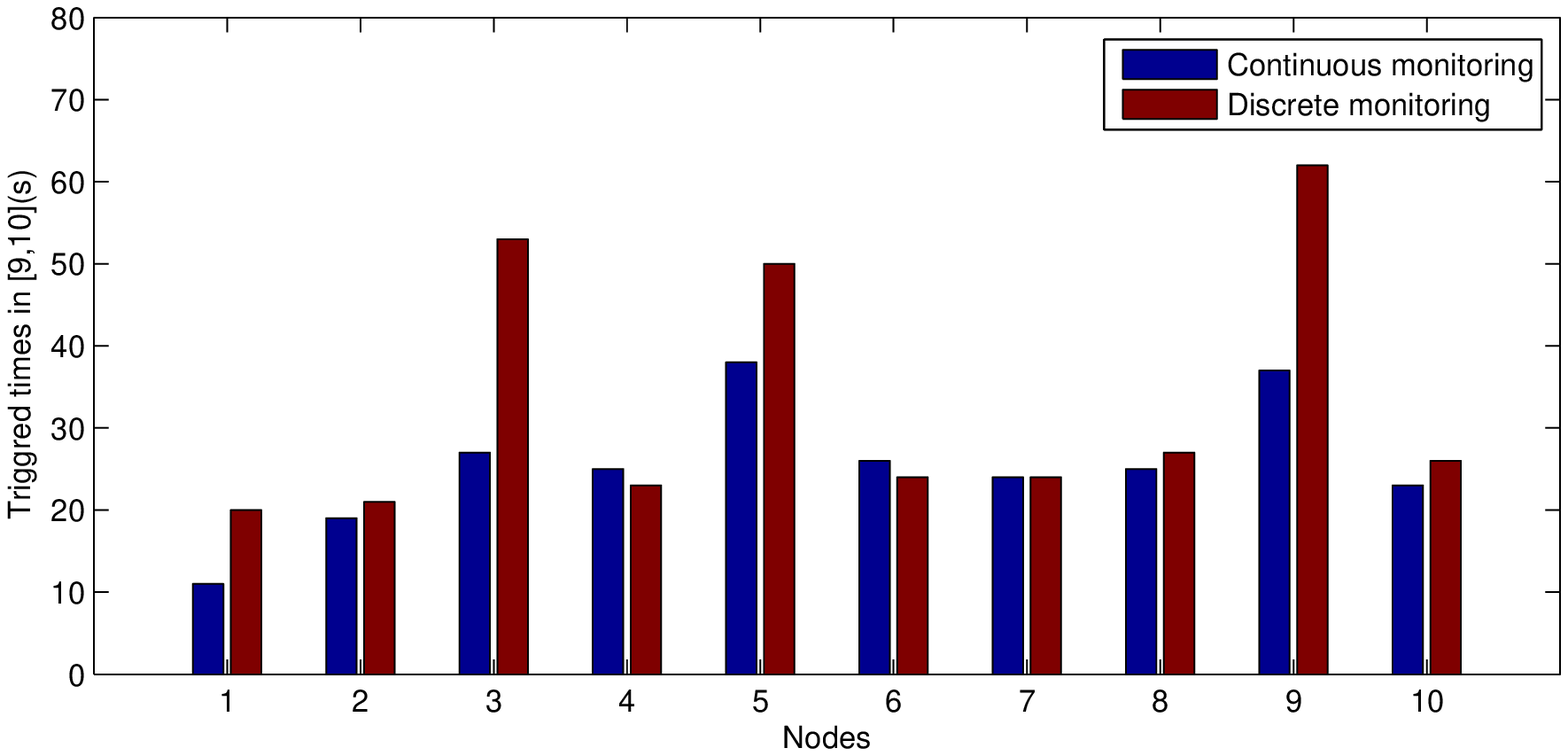} 
\label{fig_con_vs_dis2}
}
\caption{Histogram of triggering times of each node in $[9,10]$s. Under rules (a) (\ref{event1}) and (\ref{event3}), (b) (\ref{event2}) and (\ref{event4}).  }
\end{figure}
Furthermore, as shown in Figs.\ref{fig_con_vs_dis1}-\ref{fig_con_vs_dis2}, the events of updating the diffusion and pinning terms in the discrete monitoring strategy is much more than the continuous monitoring strategy, as we expected as Table \ref{table1}. As a trade-off, the performance of the discrete monitoring in terms of convergence rate of $V(t)$ is higher than the continuous monitoring, as shown by Fig.\ref{fig_variation_v}.Further, it can be seen from Fig. \ref{fig_variation_v} that the rules (\ref{event1}) and (\ref{event3}) have higher convergence rates than (\ref{event2}) and (\ref{event4}) and close to the original coupled system with simultaneous diffusion and pinning, as a reward of high event frequency of updating diffusion and pinning terms, as shown in Figs.\ref{fig_con_vs_dis1}-\ref{fig_con_vs_dis2}.

\section{Strength, limitations of the work and orients of future research}
Event-triggered algorithm is a new issue in the coordination control. Despite attracting increasing interests recently, there are a small number of papers, for instance \cite{Gao}-\cite{ZLiu}, which were concerned with pinning control of networks with event-triggered algorithm. Moreover, all of them can not handle the scenario considered in this paper. \cite{Gao} studied distributed event-triggered mechanism for pinning control of networks with {\em static} topology, while in our paper, we studied pinning networks with Markovian switching topologies and Markovian switching pinned node set.
\cite{Alderisio,ZLiu} investigated event-triggered pinning control of networks with static and switching topologies. However, from the sufficient conditions for complete synchronization that were given in \cite{Alderisio} and \cite{ZLiu}, the pinned coupled system with each possible topology among the switching topologies should be able to stabilize the coupled system; in comparison, in our work, there may exist some network topology and pinned node set in the state space of the Markovian chain that cannot stabilize the coupled subsystem. Moreover, they have not considered the switching of pinned node set, which were taken into consideration in the present paper.

However, there are a few limitations of the present methods. First, the present study assumes there is no delay in information transitions. But real networks  may have limit bandwidth limitation that will cause delays in message delivery. An interesting future research may take the time-delays into consideration. Second, in this work, we assume that the possible graph topologies and corresponding coupling matrices are already given and induced by a Markovian chain. It is sufficient for constructing a Lyapunov function to prove the stability of system. But in reality, the coupling  weights of every possible graph topology may be time-varying. it is an important issue, and will be addressed in the near future.

Third, this work mainly shows that if the linearly coupled system with persistent diffusion and control can be stabilized, then the proposed event-triggered rules can stabilize the system, too. It is important to extend the model to nonlinear cases. In \cite{lxwang1}-\cite{lxwang3}, the authors proved that fuzzy systems are universal approximators  for nonlinear dynamic systems. Hence, applying event-triggered strategies to fuzzy systems can be seen as a modest step. For this issue, we refer readers to  \cite{Qiu1}-\cite{Qiu3} . Recently, \cite{Peng} proposed a centralized event-triggered communication scheme for networked Takagi-Sugeno fuzzy systems, while distributed event-triggered algorithms for fuzzy systems are absent. This also leads to  interesting orients of our future work.

\section{Conclusions}
In this paper, event-triggered configurations and pinning control are employed to realize stability in linearly coupled dynamical systems with Markovian switching in both coupling matrix and pinned node set. Two monitoring scenarios are considered.
For continuous monitoring, each node observes its neighborhood's state and the target's state (if it is pinned) in an instantaneous way to determine the next triggering event time for updating state information. Instead, for discrete monitoring, each node can only obtain the state information at the event time or switching time of the underlying Markov chain to predict the next triggering event time for updating state information.
Once an event for a node is triggered, the diffusion coupling term and feedback control term of this node is updated.
Event triggering criteria are derived for each node that can be computed in a parallel way.
For both scenarios, it is proved that the coupled system can realize stability and the rule of piece-wise constant diffusion and pinning (if pinned) terms can efficiently reduce the computation load of the networked system, in comparison to the original coupled system. In addition, the discrete monitoring strategy can also reduce the communication load as well. Zeno behaviors can be proved excluded by proving the positivity of the lengths of the inter-event time intervals for some rules. Simulations are given to verify these theoretical results.

\section*{Acknowledgements}
This work is jointly supported by the National Natural Sciences
Foundation of China under Grant (Nos. 61273211 and 61273309), the Marie
Curie International Incoming Fellowship from the European Commission
(FP7-PEOPLE-2011-IIF-302421), the Program for New Century Excellent Talents
in University (NCET-13-0139), and the Programme of Introducing Talents of Discipline to
Universities (B08018).


\begin{thebibliography}{99}
\bibitem{Liu12}
H. Y. Liu, G. M. Xie and L. Wang, Containment of linear multi-agent systems under general interaction topologies, System \& Control Letters, 61(4) (2012): 528--534.

\bibitem{Hgenii}
C. Hgenii, Horoloquim oscilatorium (Aqud F. Muget, Parisiis), 1673.
\bibitem{Wu}
C. W. Wu and L. O. Chua, Synchronization in an array of linearly coupled dynamical systems,  IEEE Transactions on Circuits and Systems I: Fundamental Theory and Applications, 42(8) (1995): 430-447.
\bibitem{Belykh}
V. N. Belykh, I. V. Belykh and M. Hasler, Connection graph stability method for synchronized coupled chaotic systems, Physica D: nonlinear phenomena, 195(1) (2004): 159-187.
\bibitem{Cao}
J. Cao, P. Li and W. Wang, Global synchronization in arrays of delayed neural networks with constant and delayed coupling, Physics Letters A, 353(4) (2006): 318-325.
\bibitem{Lu04}
W. Lu and T.  Chen, Synchronization analysis of linearly coupled networks
of discrete time systems, Physica D: Nonlinear Phenomena, 198(1) (2004): 148-168.

\bibitem{Lu04a}
W. Lu and T. Chen, Synchronization  of Coupled Connected Neural
Networks With Delays, IEEE Transactions on Circuits and Systems Part 1: Regular Papers, 51(12) (2004): 2491-2503.
\bibitem{Lu06}
W. Lu and T. Chen, New approach to synchronization analysis of linearly coupled ordinary differential equations, Physica D: Nonlinear Phenomena, 213(2) (2006): 214-230.

\bibitem{Lu07}
W. Lu and T. Chen, Global Synchronization of
Discrete-Time Dynamical Network With a Directed Graph, IEEE Transactions on  Circuits and Systems II: Express Briefs, 54(2) (2007): 136-140.

\bibitem{Xiang}
J. Xiang and G. Chen, On the V-stability of complex dynamical networks, Automatica, 43(6) (2007): 1049-1057.
\bibitem{Porfiri}
M. Porfiri and M. di Bernardo, Criteria for global pinning-controllability of complex networks, Automatica, 44(12) (2008): 3100-3106.
\bibitem{Wang06}
W. Wang and J. Slotine, A theoretical study of different leader roles in networks, IEEE Transactions on Automatic Control, 51(7) (2006): 1156-1161.

\bibitem{Wang02}
X. Wang and G. Chen, Pinning control of scale-free dynamical network, Physica A: Statistical Mechanics and its Applications, 310(3) (2002): 521-531.

\bibitem{Li04}
X. Li, X. Wang and G. Chen, Pinning a complex dynamical network to its equilibrium, IEEE Transactions on Circuits and Systems I: Regular Papers, 51(10) (2004): 2074-2087.

\bibitem{LLR}
W. Lu, X. Li and Z. Rong, Global stabilization of complex networks with digraph topologies via a local pinning algorithm, Automatica, 46(1) (2010): 116-121.

\bibitem{Yu}
W. Yu, G. Chen and J. L\"u, On pinning synchronization of complex dynamical networks, Automatica, 45(2) (2009): 429-435.

\bibitem{Chen07}
T. Chen, X. Liu and W. Lu, Pinning complex networks by a single controller, IEEE Transactions on Circuits and Systems I: Regular Papers, 54(6) (2007): 1317-1326.
\bibitem{han}
Y. Han, W. Lu and T. Chen, Pinning dynamic systems of networks with Markovian switching couplings and controller node set, Systems \& Control Letters, 65 (2014): 56-63.


\bibitem{Astrom}
K. J. \AA str\"om and B. Bernhardsson, Comparison of Riemann and Lebesgue sampling for first order stochastic systems,
Proceedings of the 41st IEEE Conference on Decision and Control(2002): 2011-2016.


\bibitem{Tabuada}
P. Tabuada, Event-triggered real-time scheduling of stabilizing control
tasks, IEEE Transactions on Automatic Control, 52(9) (2007): 1680-1685.


\bibitem{Wang11}
X. Wang and M. D. Lemmon, Event-triggering in distributed networked control systems, IEEE Transactions on Automatic Control, 56(3) (2011): 586-601.


\bibitem{Dimarogonas}
D. V. Dimarogonas, E. Frazzoli and K. H. Johansson, Distributed event-triggered control for multi-agent systems, IEEE Transactions on Automatic Control, 57(5) (2012): 1291-1297.

\bibitem{Johannesson}
E. Johannesson, T. Henningsson and A. Cervin, Sporadic control of first-order linear stochastic systems, Hybrid Systems: Computation and Control, Lecture Notes in Computer Science, 4416 (2007): 301-314.

\bibitem{Rabi}
M. Rabi, K. H. Johansson and M. Johansson, Optimal stopping for event-triggered sensing and actuation, 47th IEEE Conference on  Decision and Control, 2008 (IEEE CDC 2008): 3607-3612.


\bibitem{Seyboth}
G. S. Seyboth, D. V. Dimarogonas and K. H. Johansson, Event-based broadcasting for multi-agent average consensus, Automatica, 49(1) (2013): 245-252.


\bibitem{YuHan}
H. Yu and P. J. Antsaklis, Output synchronization of multi-agent systems with event-driven communication: communication delay and signal quantization, ISIS (2011): 001.

\bibitem{Yi1}
X. Yi, W. L. Lu, T. P. Chen, Event-triggered Consensus for Multi-agent Systems with Asymmetric and Reducible Topologies, 2014, arXiv:1407.1377


\bibitem{Gao}
L. Gao, X. Liao and H. Li,
Pinning controllability analysis of complex networks with a distributed event-triggered mechanism, IEEE Transactions on Circuits and Systems II: Express Briefs, 61(7) (2014):.

\bibitem{Alderisio}
F. Alderisio, Pinning Control of Networks: an Event-Triggered Approach, Master Thesis, KTH Royal Institute of Technology, 2013.





\bibitem{ZLiu}
Z. Liu and Z. Chen,
Reaching Consensus in Networks of Agents via Event-triggered Control
Journal of Information \& Computational Science, 8(3) (2011): 393-402.

\bibitem{Mao06}
X. Mao and C. Yuan, Stochastic differential equations with Markovian switching, London: Imperial College Press, 2006.

\bibitem{Joh}
K. H. Johansson, M. Egerstedt, J. Lygeros, and S. S. Sastry,
On the regularization of zeno hybrid automata,
Systems \& Control Letters, 38(3) (1999): 141-150.

\bibitem{Gron}
T. H. Gronwall, Note on the derivatives with respect to a parameter of the solutions of a system of differential equations, Ann. of Math. 20(2) (1919): 292-296.

\bibitem{Bell}
R. Bellman, The stability of solutions of linear differential equations, Duke Math. J. 10(4) (1943): 643-647.

\bibitem{chua}
T. Matsumoto, L. O. Chua and M. Komuro, The double scroll, IEEE Transactions on  Circuits and Systems, 32(8) (1985): 797-818.

 \bibitem{lxwang1}
L. X. Wang and J. M. Mendel, Generating fuzzy rules by learning from examples, IEEE Transactions on Systems, Man and Cybernetics, 22.6 (1992): 1414-1427.

\bibitem{lxwang2}
L. X. Wang and J. M. Mendel, Generating fuzzy rules from numerical data, with applications, Signal and Image Processing Institute, University of Southern California, Department of Electrical Engineering-Systems, 1991.


\bibitem{lxwang3}
L. X. Wang and J. M. Mendel, Back-propagation fuzzy system as nonlinear dynamic system identifiers, IEEE International Conference on Fuzzy Systems 1992: 1409-1418, San Diego, USA.

\bibitem{Qiu1}
J. Qiu, G. Feng and J. Yang, A new design of delay-dependent robust {$\mathcal H_{\infty}$} filtering for discrete-time T-S fuzzy systems with time-varying delay, IEEE Transactions on Fuzzy Systems, 17(5) (2009): 1044-1058.


\bibitem{Qiu2}
J. Qiu, G. Feng and H. Gao, Fuzzy-model-based piecewise {$\mathcal H_{\infty}$} static output feedback controller design for networked nonlinear systems, IEEE Transactions on Fuzzy Systems,18(5) (2010): 919-934.

\bibitem{Qiu3}
J. Qiu, G. Feng and H. Gao,
Static output feedback {$\mathcal H_{\infty}$} control of continuous-time T-S fuzzy affine systems via piecewise Lyapunov functions, IEEE Transactions on Fuzzy Systems, 21(2) (2013): 245-261.

\bibitem{Peng}
C. Peng, Q. Han and D. Yue, To transmit or not to transmit: a discrete event-triggered communication scheme for networked Takagi-Sugeno fuzzy systems, IEEE Transactions on Fuzzy systems, 21(1)(2013): 164-170.



\end{thebibliography}
\end{document}